\documentclass[11pt]{article} 
\usepackage[english]{babel}
\usepackage{amsmath,amssymb,amsfonts,amsthm}
\usepackage{amscd,amsfonts,mathrsfs} 
\usepackage[bookmarks=false]{hyperref}
\usepackage{cite}

\usepackage{graphicx}
\usepackage{mathabx} 

\usepackage{xcolor}
\newcommand{\R}{\mathbb R}
\newcommand{\Z}{\mathbb Z}
\newcommand{\N}{\mathbb N}
\newcommand{\CC}{\mathbb C}
\newcommand{\const}{\mathrm{const}}

\newcommand{\Sym}{\mathrm{Sym}}
\newcommand{\rank}{\mathrm{rank}\,}
\newcommand{\id}{\mathrm{id}}

\theoremstyle{plain}
\newtheorem{theorem}{Theorem}[section]

\newtheorem{proposition}[theorem]{Proposition}
\newtheorem{corollary}[theorem]{Corollary}
\theoremstyle{definition}
\newtheorem{definition}[theorem]{Definition}
\newtheorem{remark}[theorem]{Remark}

\begin{document}

\title
{Symplectic classification for universal unfoldings \\ of $A_n$ singularities in integrable systems}

\date{}

\author{Elena A. Kudryavtseva} 

\maketitle
 
\begin{abstract}
In the present paper, we obtain real-analytic symplectic normal forms for integrable Hamiltonian systems with $n$ degrees of freedom near singular points having the type ``universal unfolding of $A_n$ singularity'', $n\ge1$ (local singularities), as well as near compact orbits containing such singular points (semi-local singularities). 
We also obtain a classification, up to real-analytic symplectic equivalence, of real-analytic Lagrangian foliations in saturated neighborhoods of such singular orbits (semi-global classification). 
These singularities (local, semi-local and semi-global ones) are structurally stable.
It turns out that all integrable systems are symplectically equivalent near their singular points of this type (thus, there are no local symplectic invariants).
A complete semi-local (respectively, semi-global) symplectic invariant of the singularity 
is given by a tuple of $n-1$ (respectively $n-1+\ell$) real-analytic function germs in $n$ variables, where $\ell$ is the number of connected components of the complement of the singular orbit in the fiber.
The case $n=1$ corresponds to non-degenerate singularities (of elliptic and hyperbolic types) of one-degree of freedom Hamiltonians; their symplectic classifications were known.
The case $n=2$ corresponds to parabolic points, parabolic orbits and cuspidal tori, and the case $n\ge3$ --- to their higher-dimensional analogs.

\hspace{-5.7mm} \textit{Keywords}: 
integrable Hamiltonian systems, universal unfolding of $A_n$ singularity, symplectic invariants, symplectic classification, structurally stable singularities, period mapping.
\end{abstract}

\section{Introduction and summary of the results}

In this work, we study the problem of symplectic classification of integrable systems. Recall that an integrable system is specified by a triple $(M, \omega, F)$, where $(M,\omega)$ is a symplectic $2n$-manifold and
$$
\mathcal F = (F_1, \ldots, F_n) \colon M \to \mathbb R^n
$$
is an \textit{integral} (also called an \textit{energy-momentum}) \textit{map}, 
consisting of $n$ almost everywhere independent functions $F_i$ 
that pairwise Poisson commute:
$$
\omega(X_{F_i}, X_{F_j})= 0 \quad \mbox{for all} \quad i, j = 1, \ldots, n,
$$ 
where $X_{F_i}$ is defined by the rule $\omega(X_{F_i}, \cdot) = - \textup{d}F_i.$ 

We will assume that the symplectic manifold and the energy-momentum map are real-analytic (note that the $C^\infty$-smooth case is studied in~\cite{KudryavtsevaMartynchuk2021,KudryavtsevaMartynchuk2023,MartynchukVu-Ngoc2023}).
The function $F_1$ is typically the Hamiltonian $H$ of the system, generating the dynamics via the Hamiltonian vector field $X_H$, and the number $n$ is by definition the number of degrees of freedom. 
It will be more convenient for us to assume that the Hamilton function is $H=F_n$.
We will sometimes write $\mathcal F = (H, F_1, \ldots, F_{n-1})$ regarding $H=F_n$ as $F_0$.

A point $P\in M$ is called {\em singular} if $\rank(\textup{d}\mathcal F(P))<n$; the integer $n-\rank(\textup{d}\mathcal F(P))$ is called {\em corank} of the singular point $P$. The set of singular values 
\begin{equation} \label {eq:Sigma}
\Sigma := \{\mathcal F(P) \mid \rank(\textup{d}\mathcal F(P))<n\}
\end{equation}
is called the {\em bifurcation diagram}.
The integral map $\mathcal F$ naturally gives rise to a (singular) Lagrangian foliation (the {\em Liouville foliation}) on $M$ whose fibers are connected components of the common level sets 
$$
\mathcal F^{-1}(f) := \{F_1 = f_1, \ldots, F_n = f_n\}, \quad f = (f_1, \ldots, f_n) \in \mathbb R^n.
$$
One can also write this foliation as the quotient map $\pi \colon M \to B,$ where $B$ is the set of connected components of $\mathcal F^{-1}(f)$, $f \in \mathbb R^n,$ equipped with the quotient topology \cite{Fomenko1988}. The space $B$ is usually referred to as the \textit{bifurcation complex} (or the \textit{unfolded momentum domain}) of the system. Note that all of the fibers of $\pi \colon M \to B$ are invariant under the flows of the Hamiltonian vector fields $X_{F_1}, \ldots, X_{F_n}.$

The problem of symplectic classification of integrable systems amounts to classifying the corresponding Liouville foliations $\pi \colon M \to B$ up to a `symplectic equivalence'. Usually, one considers 
`fiberwise' (= \textit{left-right}) \textit{symplectic equivalence}, where two such foliations are called equivalent if they are related by a symplectomorphism sending fibers to fibers (note that right symplectic equivalence is studied in \cite{DeVerdiere1979, Delzant1988, Francoise1988, Bolsinov2018, MartynchukVu-Ngoc2023}).

\begin{definition} \label{definition/equivalence}
Two integrable Hamiltonian systems $(M, \omega, \mathcal F)$ and $(\tilde M, \tilde{\omega}, \tilde{\mathcal F})$ are called
\textit{(left-right)
\footnote{
see for instance \cite{Hanssmann2006}.} symplectically equivalent} if there exists a real-analytic symplectomorphism 
$$
\Phi \colon  (M,\omega) \to (\tilde M, \tilde{\omega})
$$ 
and a real-analytic diffeomorphism $\phi \colon \R^n \to \R^n$ such that 
\begin{equation} \label {eq:rigid}
\phi \circ \mathcal F = \tilde{\mathcal F} \circ \Phi.
\end{equation}
\end{definition}

Finally, we note that symplectic classification of integrable systems can be approached from the 

i) local (in a neighborhood of a singular point),

ii) semi-local (in a neighborhood of a singular orbit),

iii) semi-global (in a saturated neighborhood of a singular fiber)

iv) or global perspectives. \\
For a more in depth introduction to this and related problems, we refer the reader to
\cite{Bolsinov2004, Bolsinov2006}.

A natural set of symplectic invariants of an integrable system is given by its action variables. Assume for the moment that all of the fibers $\mathcal F^{-1}(f)$ are compact and connected. Then the action variables can be defined
by the Mineur--Arnol'd formula
\begin{equation} \label {eq:action}
I_i = \dfrac{1}{2\pi} \int_{\gamma_i} \alpha, \quad \textup{d}\alpha = \omega, 
\end{equation}
where  
$\gamma_i$ are independent homology cycles on a regular fiber $F^{-1}(f)$ continuously depending on $f$ (recall that regular, compact and connected fibers $F^{-1}(f)$ are $n$-dimensional tori by the Arnol'd--Liouville theorem \cite{Arnold1968, Arnold1978}). Any symplectomorphism $\Phi \colon  M \to \tilde M$ that respects the foliations induced by the integral maps $\mathcal F$ and $\tilde{\mathcal F}$ must send the set of action variables
of $(M, \omega, \mathcal F)$ to some set of action variables of $(\tilde M, \tilde{\omega}, \tilde{\mathcal F})$. Therefore, the action variables are ``rigid'' first integrals in the sense that they are uniquely defined (up to affine $\mathbb R^n \rtimes \textup{SL}(n,\mathbb Z)$ transformations) on a neighborhood of a regular fiber; in other words, they are ``symplectic invariants''.
We say that action variables form a \emph{complete set} of symplectic invariants for a given class of integrable systems if (see also \cite{Bolsinov2018_2} and Corollary~\ref{cor:An:semiloc} below) for every two integrable systems $\mathcal F^1 \colon M^1 \to \mathbb R^n$ and $\mathcal F^2 \colon M^2 \to \mathbb R^n$ from this class, 
the existence of a diffeomorphism $\phi \colon B_1 \to B_2$ between the bifurcation complexes $B_1$ and $B_2$ of $\mathcal F_1$ and $\mathcal F_2$ that preserves the action variables implies that
$\phi$ can be lifted to a foliation-preserving symplectomorphism $\Phi \colon M_1 \to M_2.$

It is known that in many cases, action variables form a complete set of symplectic invariants in the above sense (but not always;
see for example the classical work by Duistermaat \cite{Duistermaat1980}). The well
known such examples include 
the so-called \textit{toric systems} (Delzant's theorem \cite{Delzant1988}), simple Morse functions $H \colon M^2 \to \mathbb R$ on compact symplectic $2$-surfaces \cite{Vu-Ngoc2011, Dufour1994} and simple focus-focus singularities\footnote{This property of simple focus-focus singularities can be derived from \cite[Theorem~2.1 and Def.~3.1]{Vu-Ngoc2003} via the following fact, which can be derived from explicit formulae in \cite[Sec.~3]{Vu-Ngoc2003}: if two simple focus-focus singularities have the same action variables then they have the same Eliasson momentum map up to flat functions.}. We note that simple focus-focus singularities 
are symplectically classified in \cite{Vu-Ngoc2003}.
Moreover, there is also a general symplectic classification of semi-toric systems \cite{Pelayo2009, Pelayo2011}.

In connection to the above list of examples, \cite{Bolsinov2018_2} posed the problem of proving the `completeness' property of action variables for a larger class of (singularities of) integrable systems. For this problem, most progress has been made in the analytic category.
In the context of unfolded $A_n$ singularities (which are elliptic or hyperbolic singular points if $n=1$, parabolic singularities if $n=2$), we first of all mention  \cite{VarchenkoGivental1982} (see also \cite{ColindeVerdiere2003, Garay2004}), which implies that, in a neighbourhood of an unfolded $A_n$ singular point, no symplectic invariants exist (except for some discrete parameters) and the singular Lagrangian foliation can be brought to a (left-right) normal form by means of a real-analytic symplectomorphism; see Theorem~\ref{thm:An:loc}. In \cite{Bolsinov2018} it was shown that, for parabolic orbits and cuspidal tori, which are the simplest examples of degenerate singularities of two-degree of freedom systems, action variables are the only symplectic invariants in the real-analytic category.
One more case where such a situation occurs (in the real-analytic category) are non-degenerate semilocal singularities satisfying a connectedness condition \cite{KudryavtsevaOshemkov2021}, as has been established in \cite{KudryavtsevaOshemkov2022}.

In many cases, action variables generate an effective Hamiltonian $T^n$-action near a compact orbit.
Such examples include free $T^n$-actions near regular compact orbits (due to the Liouville--Arnol'd theorem), effective $T^n$-actions near elliptic compact orbits. 
In real-analytic case, for non-degenerate singular compact orbits (and only for them) there exists an effective Hamiltonian $T^n$-action generated by first integrals on a small open complexification of the orbit (``toric singularities'').
In all these examples, the $T^n$-action is symplectically rigid (i.e., persistent under integrable perturbations).

For degenerate singularities, the situation is more complicated, since one should take into account action variables that do not generate torus action. Many degenerate singularities still admit effective Hamiltonian $T^{n-1}$-actions near compact singular orbits 
(``semi-toric singularities'' \cite {KudryavtsevaLerman2024}); e.g.\ corank-1 singular orbits that lie in fibers $\mathcal L$ such that $\dim\mathcal L\le n$ \cite{Zung2000} (e.g.\ parabolic orbits with resonances), integrable Hamiltonian Hopf bifurcation; 
see (non-symplectic) normal forms of singular orbits of corank $2$ \cite {KudryavtsevaLerman2024} and corank $1$ \cite {AKKO2024} in integrable systems with $n\le 3$ degrees of freedom.
In real-analytic case, sufficient conditions are known for the existence of a locally-free $T^r$-action near a compact $r$-dimensional orbit \cite{Zung2006}, as well as for the existence of an effective torus-action near a singular point, and for persistence of such actions under integrable perturbations \cite {Kudryavtseva2020}. 
The author classified \cite {Kudryavtseva2020} Hamiltonian $T^k$-actions on small open complexifications of compact singular orbits.

All this leaves open the question of what are symplectic invariants in a neighbourhood of a degenerate singular point (resp., orbit, fiber) that classify such singularities up to symplectic equivalence, and whether rigidity holds for a local, semi-local or semi-global singularity (by rigidity, we mean that the mapping germ \eqref {eq:rigid} at a singular point or orbit is uniquely defined, i.e., does not depend on $\Phi$ and $\phi$).
Our purpose is to study these questions for singularities that are structurally stable (since they appear in many integrable systems).

The present work solves these questions in real-analytic case for an infinite series of corank-1 singularities having the type ``universal unfolding of $A_n$ singularity'' (or ``unfolded $A_n$ singularity'' for short), $n\ge1$.
As it turns out, a mapping germ $(\R^n,0)\to(\R^{n-1},0)$ satisfying some natural conditions classifies the singularity at the orbit (i.e., semi-local singularity) up to real-analytic symplectic equivalence. Similarly, a mapping germ $(\R^n,0)\to(\R^{n-1+\ell},0)$ 
classifies the singularity at the fiber (i.e., semi-global singularity) up to real-analytic symplectic equivalence, where $\ell$ is the number of connected components of $\mathcal L\setminus\mathcal O$, $\mathcal O$ is the orbit and $\mathcal L$ is the fiber containing $\mathcal O$ (in fact, there are also ``discrete'' topological invariants, namely: a local invariant $\eta=\pm1$ for odd $n$, and a semi-global invariant that is a permutation $\beta\in \Sym_\ell$).
This is the main contribution of the present work; see Theorems~\ref{thm:An:loc}, \ref{thm:An:semiloc}, \ref{thm:An:semiglob}, Corollary \ref {cor:An:semiloc:class},
which give a real-analytic classification of unfolded $A_n$ 
singular points, singular orbits and singular fibers up to real-analytic (left-right) symplectic equivalences.
We note that in the latter theorems the invariants are explicitly given by collections of functions in $n$ variables. We also establish properties of local action variables in the swallowtail domain (Propositions \ref {pro:inject}, \ref {pro:equivar}).
As a consequence, we establish the completeness of the action variables near the singular orbits (see Corollary \ref{cor:An:semiloc}).

As an application, consider magnetic geodesic flows on surfaces of revolution \cite {KobtsevKudryavtseva2025, KobtsevKudryavtseva2025a}.
Such a system is an integrable Hamiltonian system with 2 degrees of freedom. This system is given by a pair of functions (determining the metric and the magnetic field, respectively) in one variable, so this pair of functions surves as parameters of the system.
If this pair of functions satisfies some natural conditions of general position, then all singular points of the system are as follows: 
two center-center singular points, several 1-parameter families of saddle circles, several 1-parameter families of center circles, and several degenerate singular orbits having two types: parabolic orbits and ``asymmetric elliptic pitchforks'' \cite {KobtsevKudryavtseva2025}. 
In known mechanical systems, a pitchfork usually appears if the system is invariant under a canonical involution. However, it turns out that, in magnetic geodesic flows on surfaces of revolution, elliptic pitchforks appear that do not admit a system-preserving canonical involution near the singular orbit, provided that the system satisfies conditions of general position.
Using our symplectic normal form for unfolded $A_3$ singular points (Theorem \ref {thm:An:loc}), one can show that, for singular points of the type ``asymmetric elliptic pitchfork'', a local symplectic invariant is a function germ $\tilde J_1(\tilde H,\tilde J_2)$ at $(0,0)$ (which is identical $0$ if and only if the singularity is symmetric), while 
$\tilde J_1(0,0)\ne0$ for magnetic geodesic flows satisfying conditions of general position, which means that the pitchfork is indeed asymmetric. We plan to present these results in a future work.

\section {Parabolic singulariy ($n=2$)}

Our first Theorem~\ref{thm:An:loc} concerns singular points having the type ``universal unfolding of $A_n$ singularity''
(which are parabolic orbits if $n=2$, see e.g.\ \cite{Lerman1994, Efstathiou2012, Bolsinov2018} for a background material). 
Before stating this theorem, recall \cite{Bolsinov2018} the definition of a parabolic point 
of an integrable 2 degree of freedom Hamiltonian system.

Let $H$ and $F$ be a pair of Poisson commuting real-analytic functions on a real-analytic symplectic manifold $(M^4, \omega)$. They define a Hamiltonian $\R^2$-action (perhaps local) on $M^4$. The dimension of the $\R^2$-orbit through a point $P\in M^4$ coincides with the rank of the differential of the momentum map $\mathcal F=(H,F): M^4 \to \R^2$ at this point. We are interested in one-dimensional orbits and without loss of generality we assume that $\textup{d}F(P)\ne 0$.  Consider the restriction of $H$ onto the three-dimensional level set of $F$ through $P$, that is, $H_0:=H|_{\{F=F(P)\}}$. We assume that the rank of $\textup{d}\mathcal F$ at the point $P$ equals one. This is equivalent to any of the following:
\begin{itemize}
\item $P$ is a critical point of $H_0$;
\item there exists a unique $k\in \R$ such that $\textup{d}H(P) = k \textup{d}F(P)$, in particular, $P$ is a critical point of $F-kH$.
\end{itemize}

These properties hold true for each singular point $P$ of rank one of the momentum map $\mathcal F=(H,F)$ under the condition that $\textup{d}F(P)\ne 0$.

\begin{definition} [{\cite{Bolsinov2018}}] \label{def:parabolic}
A point $P$ (and the corresponding $\R^2$-orbit through this point) is called {\it parabolic} if the following conditions hold:
\begin{enumerate}
\item[(A1)] the quadratic differential $\textup{d}^2H^0(P)$ has rank 1; 
\item[(A2)] there exists a vector $v \in \ker \textup{d}^2H_0(P)$ such that $v^3H_0\ne 0$ (by $v^3H_0$ we mean the third derivative of $H_0$ along the tangent vector $v$ at $P$);
\item[(A3)] the quadratic differential $\textup{d}^2(H-kF)(P)$ has rank 3, where $k$ is the real number determined by the condition $\textup{d}H(P)=k\textup{d}F(P)$.
\end{enumerate}
\end{definition}

For example, consider the functions $H=p^2+q^3+Jq$ and $F=J$. They Poisson commute with respect to the canonical symplectic 2-form $\omega=\textup{d}J\wedge \textup{d}\varphi + \textup{d}p\wedge \textup{d}q$. The origin is a parabolic singular point, see Fig.~\ref {figure/all}.

\begin{figure}[htbp]
\begin{center}
\includegraphics[width=0.48\linewidth]{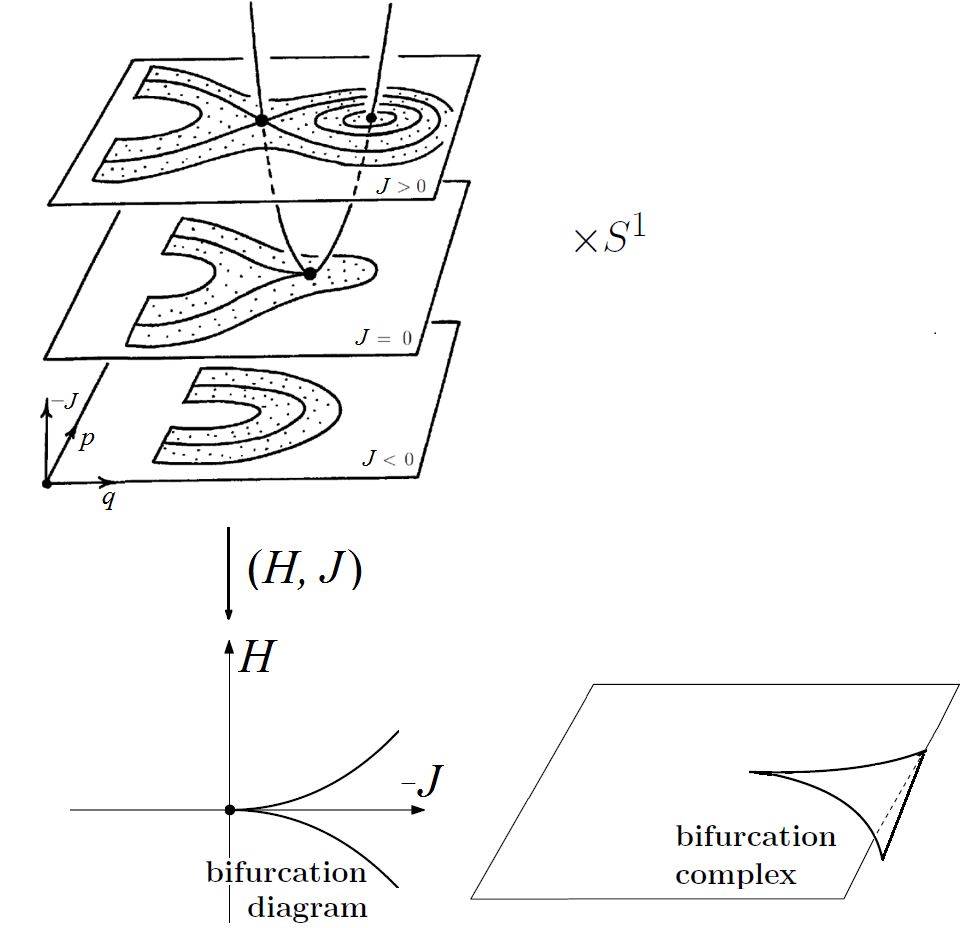} \qquad
\includegraphics[width=0.38\textwidth]{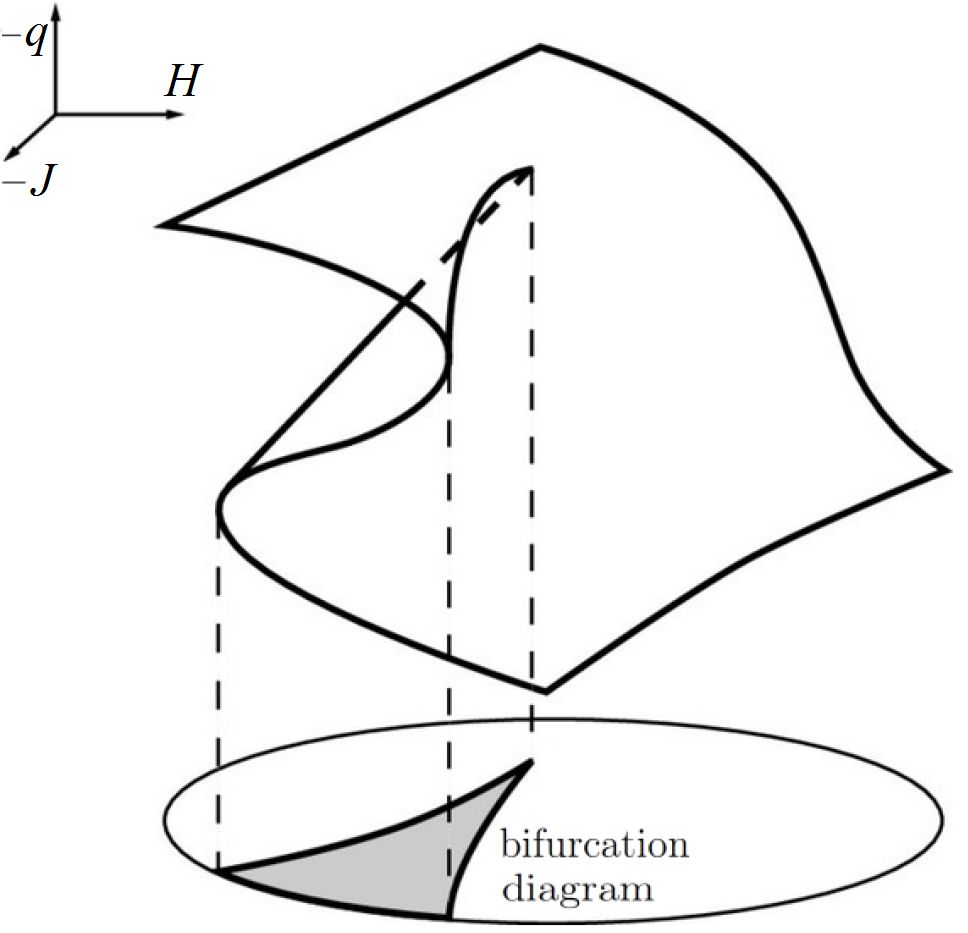} 
\end{center}
\caption{Parabolic singularity ($n=2$): the singular Lagrangian foliation (top left), the bifurcation diagram (bottom left), the bifurcation complex (bottom middle) of the energy-momentum map $(H,J)$, and the Whitney fold (right). The vanishing cycles are small loops around the `right' branch of the parabola in the top left figure. The cusp domain $D$ is a dashed area in the right figure, and the region in the bottom left figure to the right of the semicubical parabola (= the bifurcation diagram)}
\label{figure/all}
\end{figure}

\section {Universal unfolding of $A_n$ singularity} \label {sec:An}

In this section, we define a singularity of the type ``universal unfolding of $A_n$ singularity''
(or ``unfolded $A_n$ singularity'' for short)
for an integrable Hamiltonian system with $n$ degrees of freedom, $n\ge1$. 
For $n=1$, it is an elliptic or hyperbolic singularity. 
For $n=2$, it is a parabolic singularity.

Let $F_1,\dots,F_n$ (with $H=F_n$) be a collection of Poisson commuting real-analytic functions on a real-analytic symplectic manifold $(M^{2n}, \omega)$. They define a Hamiltonian $\R^n$-action (perhaps local) on $M^{2n}$. The dimension of the $\R^n$-orbit through a point $P\in M^{2n}$ coincides with the rank of  the differential of the momentum map $\mathcal F=(F_1,\dots,F_n): M^{2n} \to \R^n$ at this point. We are interested in $(n-1)$-dimensional orbits and without loss of generality we assume that $\textup{d}F_1(P),\dots,\textup{d}F_{n-1}(P)$ are linearly independent at a point $P\in M$.

Let us apply the procedure of a local symplectic reduction \cite{MarsdenWeinstein1974} by (local) Hamiltonian $\R^{n-1}$-action generated by the functions $F_1,\dots,F_{n-1}$. This procedure allows one to reduce the system to a Hamiltonian system that has one degree of freedom 
but depends on parameters. Let us describe this procedure.
Since the functions $F_1,\dots,F_{n-1}$ are in involution, the Hamiltonian vector fields $X_{F_i}$ corresponding to them commute and their flows generate a locally free action $G \times M \to M$ of the commutative group $G = \R^{n-1}$. Consider the submanifold $F_1 = F_1(P), \dots, F_{n-1} = F_{n-1}(P)$; in a neighborhood of $P$ it is a smooth $(n+1)$-dimensional disk $D^{n+1}$ containing the point $P$, and the restriction of the form $\omega$ to it has constant rank equal to $2$. It is smoothly fibered by local $(n-1)$-dimensional orbits, on each of which the form is identically equal to zero. 
Taking the quotient $D^{n+1}/G$, under which $(n-1)$-dimensional orbits are mapped to points, we obtain a symplectic $2$-dimensional local manifold with respect to the quotient 2-form $\widehat\omega$.
One can introduce local symplectic coordinates $(u,v)$ on it centred at the point $P$. Now one can transfer these coordinates to a full neighborhood $U$ of $P$ to obtain symplectic coordinates 
\begin{equation} \label {eq:coord:loc}
(F_1, \dots , F_{n-1}, \tau_1, \dots, \tau_{n-1}, u, v) \colon\, U \stackrel{\approx}{\longrightarrow} D^{n-1}\times T^{n-1} \times D^2,
\end{equation}
in which the 2-form $\omega$ reads
\begin{equation} \label {eq:Omega}
\omega = \sum_{j=1}^{n-1} \textup{d}F_j\wedge \textup{d}\tau_j + \textup{d}u\wedge \textup{d}v .
\end{equation}
The above procedure gives rise to a {\it reduced Hamiltonian system} with one degree of freedom with parameters $\varepsilon=(\varepsilon_1,\dots,\varepsilon_{n-1}) \in \R^{n-1}$ equal to the values of functions $F_1,\dots,F_{n-1}$. We will use the notation
$$
H^\varepsilon := H|_{F_1 = F_1(P)+\varepsilon_1, \dots, F_{n-1} = F_{n-1}(P)+\varepsilon_{n-1}}, \quad \varepsilon=(\varepsilon_1,\dots,\varepsilon_{n-1}) \in \R^{n-1},
$$
where $H$ is the reduced Hamiltonian.

We assume that the rank of $\textup{d}\mathcal F$ at the point $P$ equals $n-1$. This is equivalent to the following condition:
the linear part of the function $H^0$ at the point $P$ is zero, i.e., $\textup{d}H^0(P) = 0$. 
These properties hold true for each singular point $P$ of corank one of the momentum map $\mathcal F=(F_1,\dots,F_n)$ under the condition that $\textup{d}F_1(P),\dots,\textup{d}F_{n-1}(P)$ are linearly independent.

\begin{definition} \label{def:pleat}
A point $P$ (and the corresponding $\R^n$-orbit through this point) will be called  
a {\em singularity having the type of universal unfolding of $A_n$ singularity}, or just {\em unfolded $A_n$ singularity} ($n\ge1$),
if either $n=1$ and the quadratic differential $\textup{d}^2H^0(P)$ is non-degenerate, or $n\ge2$ and the following conditions hold:
\begin{enumerate}
\item[($A_n$1)] the quadratic function $\textup{d}^2H^0(P)$ has rank 1 (this implies that there exists a linear change of variables $(u,v)\to(u_1,v_1)$ such that $\textup{d}^2H^0(P) = \eta u_1^2$, $\eta=\pm1$, furthermore, by completing the square, one easily constructs a change of variables $(u_1,v_1)\to(x=u_1+v_1^2Q(v_1),y=v_1)$ for some polynomial $Q(v_1)$ of degree $\le\frac{n-3}2$ such that $\partial_{xy^{i+1}}^{(i+2)} H^0(P) = 0$ for all $i$ with $1\le i\le[\frac{n-1}2]$);
\item[($A_n$2)] the Taylor series coefficients of the function $H^0$ at $P$ satisfy the following condition:
$\partial_{y^3}^{'''} H^0(P) = \dots = \partial_{y^n}^{(n)} H^0(P) = 0$ and 
$\partial_{y^{n+1}}^{(n+1)} H^0(P) \ne 0$ (if $n=2$, then this condition means that $\partial_{y^3}^{'''} H^0(P) \ne 0$); 
\item[($A_n$3)] the square $(n-1)\times(n-1)$ matrix $\partial_{y^j\varepsilon_i}^{(j+1)} H^\varepsilon(P)$ at $\varepsilon=0\in\R^{n-1}$, $1\le i,j\le n-1$, is non-degenerate.
\end{enumerate}
Note that the local coordinates $(x,y)$ in assumptions ($A_n$2) and ($A_n$3) are those from assumption ($A_n$1).
Without loss of generality, we will assume that the value $\partial_{y^{n+1}}^{(n+1)} H^0(P)$, that is non-zero by assumption ($A_n$2), is positive (this can be achieved by replacing 
$\eta\to -\eta$ and $H\to -H$). 
If $n$ is even, we can and will assume that $\eta=+1$ (this can be achieved by replacing 
$H\to -H$ and $y\to -y$). 
If $n$ is odd, we will distinguish the {\em supercritical} ($\eta=+1$) and {\em subcritical} ($\eta=-1$) singularities and use the term {\it unfolded $A_n^\eta$ singularities}.
\end{definition}

It is an easy exercise that, for systems with $n=2$ degrees of freedom, our notion of an unfolded $A_2$ singularity (Definition \ref {def:pleat})
is equivalent to the notion of a parabolic point (Defition \ref {def:parabolic}).

It is known (see e.g.~\cite {KudLak2006}) that, under the assumptions of Definition~\ref{def:pleat}, there exists a local (non-canonical, generally speaking) change of variables 
\begin{equation} \label{eq:Phi1} \widehat\Phi = (\lambda,\mu,p,q) \colon\quad  
(\varepsilon,\tau,x,y) \ \mapsto \ (\lambda(\varepsilon),\mu(\varepsilon,\tau,x,y) ,p(\varepsilon,x,y),q(\varepsilon,x,y))
\end{equation}
near the point $P$ such that $P$ is the origin in the new variables $(\lambda,\mu,p,q)$ and
\begin{equation} \label {eq:Morse:type}
\widehat H \circ \widehat\Phi^{-1} = H_{can}^\lambda (p,q), \qquad 
H_{can}^\lambda (p,q) := \eta p^2 + q^{n+1} + \sum_{j=1}^{n-1} \lambda_j q^j , 
\qquad 
\frac{\partial}{\partial\mu_j}
= X_{\lambda_j}, \quad 1\le j\le n-1,
\end{equation}
where 
$\varepsilon=(\varepsilon_1,\dots,\varepsilon_{n-1})$, 
$\tau=(\tau_1,\dots,\tau_{n-1})$, 
$\lambda=(\lambda_1,\dots,\lambda_{n-1})$, 
$\mu=(\mu_1,\dots,\mu_{n-1})$, 
$\lambda(\varepsilon) = (\lambda_1(\varepsilon),\dots,\lambda_{n-1}(\varepsilon))$,
$\widehat H = H - \lambda_0(\varepsilon)$ for some smooth functions $\lambda_j(\varepsilon)$, $0\le j\le n-1$, where the change $\varepsilon\mapsto\lambda(\varepsilon)$ of small parameters is a local diffeomorphism, due to condition (A$_n$3) in Definition \ref {def:pleat}.
Moreover, if the system is real-analytic, then the above change of variables and the function $\lambda_0(\varepsilon)$ are real-analytic too. Consider the new first integrals 
$$
(\widehat H = H-\lambda_0(F_1,\dots,F_{n-1}), \lambda_1 (F_1,\dots,F_{n-1}),\dots,\lambda_{n-1} (F_1,\dots,F_{n-1}))
= \widehat\phi\circ\mathcal F
$$
obtained from $(H, F_1,\dots,F_{n-1})=\mathcal F$ by means of the diffeomorphism germ
\begin{equation} \label{eq:phi1}
\widehat\phi = (\widehat h,\lambda) \colon\, (\R^n,0) \to (\R^n,0), \qquad 
(h,\varepsilon) \ \mapsto \ (h-\lambda_0(\varepsilon),\lambda(\varepsilon)).
\end{equation}
We see that these changes bring the pair $(\omega,\mathcal F)$ to a preliminary normal form 
\begin{equation} \label {eq:Morse:type:}
(\widehat\Phi^{-1})^*\omega = \sum_{j=1}^{n-1} \textup{d} \lambda_j \wedge \textup{d}\mu_j + \textup{d}\alpha, \qquad 
\widehat\phi \circ \mathcal F \circ \widehat\Phi^{-1}
= \mathcal F_{can} := 
(H_{can}^\lambda(p,q), \lambda)
\end{equation}
where $\alpha$ is a 1-form in the $(n+1)$-dimensional space $(\lambda_1,\dots,\lambda_{n-1},p,q)$ such that $d\alpha|_{\lambda=0} = g(p,q)\textup{d}p\wedge \textup{d}q$, $g(0,0)>0$ (the latter inequality can be achieved by replacing $p\to-p$).

The standard Hamilton function $H_{can}^\lambda(p,q)$ in \eqref {eq:Morse:type} 
is known as a universal unfolding of $A_n$ singularity.
The standard momentum mapping $\mathcal F_{can}$ in \eqref {eq:Morse:type:}
corresponds to the {\it ``generalized Whitney mapping''}
$(\lambda_1,\dots,\lambda_{n-1},q) \mapsto 
(\lambda_1,\dots,\lambda_{n-1}, q^{n+1} + \sum_{j=1}^{n-1} \lambda_j q^j)$ 
\cite[Sec.~2.5]{Arnold2012}. These mappings with $n=2,3$ yield two famous Whitney singularities: the fold catastrophe ($n=2$, Fig.~\ref {figure/all}, right) and the dual cusp catastrophe ($n=3$, Fig.~\ref {fig:swallow}).

\begin{figure}[htbp]
\begin{center}
\includegraphics[width=0.32\linewidth]{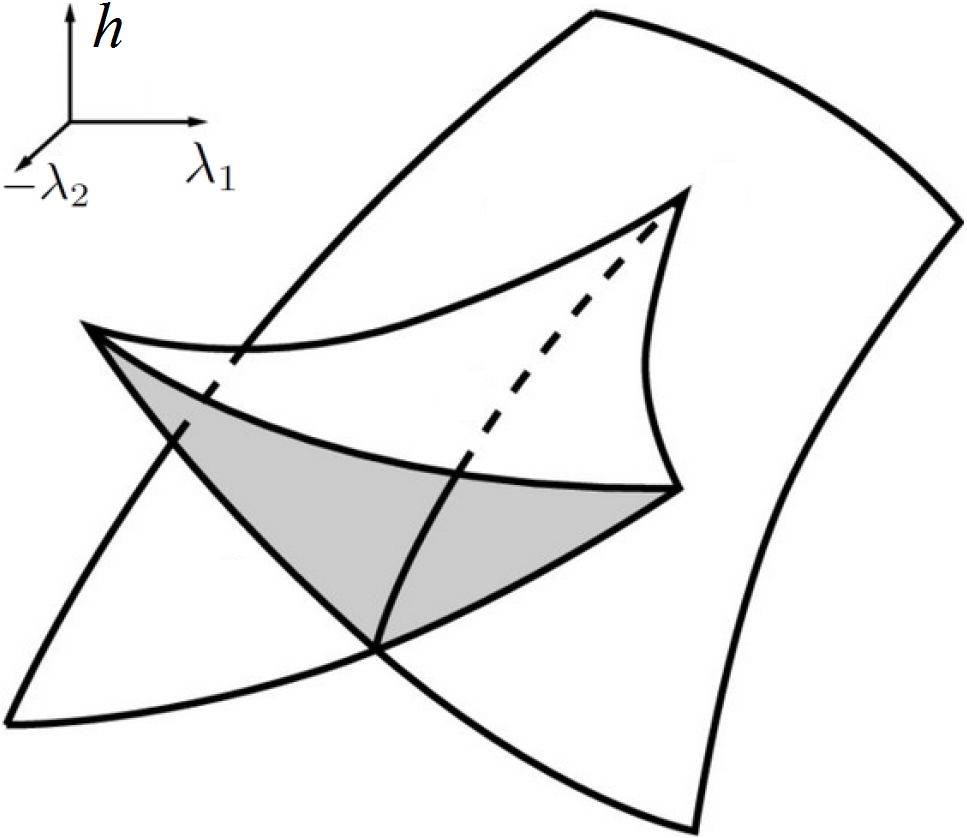}
\end{center}
\caption{The case $n=3$: the swallowtail surface $\Sigma$ (= the bifurcation diagram) and the swallowtail domain $D$ (dashed)}
\label{fig:swallow}
\end{figure}

\section {Local symplectic normal form} \label {sec:loc}

Our main result of this section studies local singularities (i.e., in a small neighbourhood of a singular point) of the Liouville foliation. It states that two local unfolded $A_n$ singularities are symplectically equivalent if and only if they are smoothly equivalent.

Having the preliminary normal form \eqref {eq:Morse:type:} in mind, we can state the following result (see Section~\ref{app:loc} for more detail and proof).

\begin{theorem} [Symplectic normal form for unfolded $A_n$ singular points] \label{thm:An:loc}
Let $P\in M$ be an unfolded $A_n$ singular point of a real-analytic integrable system $\mathcal F \colon M \to \mathbb R^n$ with $n$ degrees of freedom, $n\ge1$. Then there exist real-analytic 
local coordinates 
$\Phi=(\tilde\lambda,\tilde\mu,\tilde p,\tilde q)$ centered at $P$, in which the functions 
\begin{equation} \label {eq:loc}
\tilde H = \eta \tilde p^2 + \tilde q^{n+1} + \sum_{j=1}^{n-1} \tilde \lambda_j \tilde q^j \quad \mbox{ and } \quad \tilde{J}_j = \tilde \lambda_j, \qquad 1\le j\le n-1, 
\qquad \eta=\pm1 
\end{equation}
($\eta=+1$ if $n$ is even) 
are real-analytic functions of $(F_1,\dots,F_n)$, i.e., $(\tilde H,\tilde J_1,\dots,\tilde J_{n-1}) = \phi\circ \mathcal F$ for some real-analytic local coordinates $\phi=(\widetilde h,\tilde\lambda)$ centered at $\mathcal F(P)$, and the symplectic structure has canonical form
\begin{equation*} 
\omega = \sum_{j=1}^{n-1} \textup{d} \tilde\lambda_j \wedge \textup{d} \tilde\mu_j + 
\textup{d}\tilde p \wedge \textup{d}\tilde q .
\end{equation*}
The tuple of functions $(\tilde H,\tilde J_1,\dots,\tilde J_{n-1}) = \phi\circ\mathcal F$ is rigid in the following sense: it is uniquely defined 
when $n$ is even or $n=\eta=1$; it is uniquely defined
up to simultaneously changing the sign of the functions $\tilde J_{2i-1}$ ($i=1,\dots,m$) when $n=2m+1\ge3$ is odd, or up to changing the sign of $\tilde H$ when $n=1$ and $\eta=-1$. 
\end{theorem}

We will give a proof of Theorem \ref{thm:An:loc} in Sec.~\ref {app:loc}. Its first part 
follows from \cite{VarchenkoGivental1982}, see also \cite{ColindeVerdiere2003, Garay2004} (see also \cite{KudryavtsevaMartynchuk2023} for the case $n=2$ of parabolic points). 

Thus, by the first part of Theorem \ref {thm:An:loc}, near an unfolded $A_n$ singular point,
the pair $(\omega,\mathcal F)$ has the standard form 
\begin{equation} \label {eq:loc:}
(\Phi^{-1})^*\omega = \sum_{j=1}^{n-1} \textup{d}\tilde \lambda_j \wedge \textup{d}\tilde \mu_j + \textup{d}\tilde p\wedge \textup{d}\tilde q, \quad 
\phi\circ\mathcal F\circ\Phi^{-1} 
= \mathcal F_{can} 
:= (\eta \tilde p^2 + \tilde q^{n+1} + \sum_{j=1}^{n-1} \tilde \lambda_j \tilde q^j,\tilde \lambda_1,\dots,\tilde \lambda_{n-1}), 
\end{equation}
$\eta=\pm1$. In other words, all unfolded $A_n$ singularities of $n$-degree of freedom integrable systems are locally (i.e., in a neighborhood of the singular point) symplectically equivalent (for a given $\eta=\pm1$).
For $n=1$, these are $A_1$ singularities of $1$-degree of freedom Hamiltonians, i.e., non-degenerate singularities of elliptic ($A_1^+$) and hyperbolic ($A_1^-$) types. Their reducibility to local symplectic normal form was proved by Vey \cite{Vey1978} (see also \cite {DeVerdiere1979}), and it was extended to higher-dimensional non-degenerate rank-$0$ singularities which are direct products of elliptic and/or hyperbolic singularities \cite{Vey1978}, as well as to higher-dimensional Morse lemma with volume-preserving coordinate changes \cite{DeVerdiere1979}. 

The second part of Theorem \ref {thm:An:loc} is the rigidity property which means the following:
for any pair of coordinate changes $\widetilde\Phi,\widetilde\phi$ satisfying an analog of \eqref {eq:loc:},
\begin{itemize}
\item if $n$ is even or $n=\eta=1$, then $\widetilde\phi = \phi$;
\item if $n=2m+1$ is odd and $n\ge3$, then $\widetilde\phi = \phi$ or $\widetilde\phi = \psi\circ\phi$ where 
$\psi(\widetilde h,\tilde \lambda_1,\dots,\tilde \lambda_{n-1}) = (\widetilde h,-\tilde \lambda_1,\tilde \lambda_2,-\tilde \lambda_3,\dots,-\tilde \lambda_{n-2},\tilde \lambda_{n-1})$;
\item if $n=1$ and $\eta=-1$, then $\widetilde\phi = \phi$ or $\widetilde\phi = \psi_1\circ\phi$ where $\psi_1(\widetilde h)=-\widetilde h$.
\end{itemize}

Note that with the notation of \eqref {eq:Phi1}--\eqref {eq:Morse:type:}, the (local) bifurcation diagram of the energy-momentum map 
$\mathcal F = (H,F_1,\dots,F_{n-1})$ with $n\ge2$ is the swallowtail surface
$$
\widehat\phi^{-1}( \{(\widehat h, \lambda_1,\dots,\lambda_{n-1}) \in \mathbb R^n \mid
\mbox{the polynomial } q^{n+1} + \sum_{j=1}^{n-1} \lambda_j q^j - \widehat h \mbox{ has a multiple real root} \}).
$$
The region 
\begin{equation} \label {eq:D}
D := 
\widehat\phi^{-1}(
\{ (\widehat h,\lambda_1,\dots,\lambda_{n-1})\in\R^n \mid \exists q_1 < q_2 < \dots < q_{n+1} \ \forall q\in\R\ 
\prod_{i=1}^{n+1}(q-q_i) = q^{n+1} + \sum_{j=1}^{n-1} \lambda_j q^j - \widehat h 
\})
\end{equation}
will be called the \emph{swallowtail domain} of the singular point $P$ (see Fig.~\ref {fig:swallow} for $n=3$), or the {\it cusp domain} if $n=2$ (Fig.~\ref{figure/all}). 
The closed curves (lying in a small complexification of the singular point $P$)
\begin{equation} \label {eq:vanish}
\gamma_{h,\varepsilon,k} := 
\{
H = h,\ F_j=\varepsilon_j,\ \tau_j=0,\ 1\le j\le n-1,\
q_{k}(h,\varepsilon) \le q \le q_{k+1}(h,\varepsilon) \}, \quad 1\le k\le n,
\end{equation}
$(h,\varepsilon) = (h,\varepsilon_1,\dots,\varepsilon_{n-1}) \in D$, will be called 
\emph{vanishing cycles}, where $q_1(h,\varepsilon)<\dots<q_{n+1}(h,\varepsilon)$ denote the roots of the polynomial $q^{n+1} + \sum_{j=0}^{n-1} \lambda_j(\varepsilon) q^j - h$. 
For each $(h,\varepsilon)\in D$, the vanishing cycle \eqref {eq:vanish} 
is given by the equation $p = \pm \sqrt{\prod_{i=1}^{n+1}(q-q_i(h,\varepsilon))}$, $q_{k}(h,\varepsilon) \le q \le q_{k+1}(h,\varepsilon)$, and 
lies either in the ``real'' plane or in the ``imaginary'' plane:
\begin{equation} \label {eq:real}
\begin{gathered}
(p,q) \in \R\times\R
 \quad \mbox{ if $(-1)^{n+k}=\eta$ (``real'' plane)}, \\
(p,q) \in i\R\times\R
 \quad \mbox{ if $(-1)^{n+k}=-\eta$ (``imaginary'' plane)}.
\end{gathered}
\end{equation}
We have $n$ vanishing cycles for each $(h,\varepsilon)\in D$, and $[\frac{n+\eta}2]$ of them are ``real'' and the other
$\left\lceil \frac{n-\eta}2 \right\rceil$ are ``imaginary''.

On the swallowtail domain $D$ (see \eqref {eq:D}), consider the local action variables $I_k(h,\varepsilon)$, $1\le k\le n$ (see \eqref {eq:action}), corresponding to the family of vanishing cycles $\gamma_{h,\varepsilon,k}$, such that $I_k(h,\varepsilon)\to0$ as $(h,\varepsilon)\to \mathcal F(\mathcal O_P)$, and
\begin{equation}\label {eq:I}
\begin{gathered} 
I_k(h,\varepsilon)>0 \quad \mbox{if } (-1)^{n+k}=\eta \mbox{ (i.e., the vanishing cycle $\gamma_{h,\varepsilon,k}$ is ``real'')}, \\
-iI_k(h,\varepsilon)>0 \quad \mbox{if } (-1)^{n+k}=-\eta \mbox{ (i.e., the vanishing cycle $\gamma_{h,\varepsilon,k}$ is ``imaginary'')}
\end{gathered}
\end{equation}
(here the indicated signs of the action variables are achieved by choosing appropriate orientations of the vanishing cycles). Thus, we have the {\em local action mapping} (known as the {\em period mapping} for the 1-form $\alpha|_{\varepsilon=\const}$) on the swallowtail domain $D$:
\begin{equation} \label {eq:I:}
\begin{gathered}
I = (I_1,\dots,I_n)\colon\, D \ \to \ \Delta \approx \R^n_{>0}, \\
\Delta := \left \{ \begin{array} {ll}
\R_{>0}\times i\R_{>0} \times \R_{>0} \times i\R_{>0}\times\dots\times\R_{>0} & \mbox{if $n$ is odd, $\eta=+1$}, \\ 
i\R_{>0}\times \R_{>0} \times i\R_{>0} \times \R_{>0}\times\dots\times i\R_{>0} & \mbox{if $n$ is odd, $\eta=-1$}, \\ 
i\R_{>0}\times \R_{>0} \times i\R_{>0} \times \R_{>0}\times\dots\times \R_{>0} & \mbox{if $n$ is even},
\end{array} \right.
\end{gathered}
\end{equation}
here $d\alpha|_{\varepsilon=\const} = g(p,q,\varepsilon)dp\wedge dq$, $g(0,0,0)\ne0$, and we assume without loss of generality that $g(0,0,0)>0$.

\begin{remark} 
Let us describe the swallowtail domain \eqref {eq:D} using the elementary symmetric polynomials
$$
\sigma_k(q_1,\dots,q_{n+1})
= \sum_{1\le j_1<\dots<j_k\le n+1} q_{j_1}\dots q_{j_k}, \quad 1\le k \le n+1.
$$
Note that the swallowtail domain $\widehat\phi(D)$ admits the parametrization
$$
\vec{q} \mapsto (\widehat h,\lambda_1,\dots,\lambda_{n-1}) 
= (-\sigma_{n+1}(\vec q),\sigma_n(\vec q),\dots,\sigma_2(\vec q)),\ 
\vec{q} = (q_1,\dots,q_{n+1}) \in\R^{n+1},\ 
q_1<\dots<q_{n+1},\ 
\sigma_1(\vec q) = 0.
$$ 
We easily obtain (using the relations $\sigma_1(\vec q)=0$ and $\lambda_{n-1}=\sigma_2(\vec q) = \frac12(\sigma_1(\vec q)^2 - \sum_{j=1}^{n+1}q_j^2) = -\frac12|\vec q|^2$) the following estimates on the intersection of the swallowtail domain $\widehat\phi(D)$ 
with a small neighborhood of the origin for $n\ge2$:
\begin{equation} \label {eq:swallow}
\begin{aligned}  
\lambda_{n-1} < 0, \qquad  
\lambda_{k} &= O((-\lambda_{n-1})^{(n+1-k)/2}), \quad 1\le k \le n-2, \\ 
\widehat h  &= O((-\lambda_{n-1})^{(n+1)/2}),
\end{aligned}
\end{equation}
therefore this intersection lies in the half-space $\{\lambda_{n-1}<0\}$ and is close to the ray $\{\widehat h=\lambda_1=\dots=\lambda_{n-2}=0,\ \lambda_{n-1}<0\}$, see Fig.~\ref {figure/all} and~\ref {fig:swallow} for $n=2,3$ (for $n=1$, the swallowtail domain is a positive ray $\{\widehat h>0\}$).
\end{remark}

Having the above definitions in mind, we can state the following properties of the mapping \eqref {eq:I:}.

\begin{proposition} [Injectivity of the local action mapping for unfolded $A_n$ singular point] \label{pro:inject}
Consider a real-analytic integrable system $(M,\omega,\mathcal F)$ having unfolded $A_n^\eta$ singularity at a point $P$ ($n\ge1$, $\eta=\pm1$). Consider the local action variables $I_1,\dots,I_n$ on the swallowtail domain $D$ corresponding to the families of vanishing cycles $\gamma_{h,\varepsilon,k}$, $1\le k\le n$, $(h,\varepsilon)\in D$.
Then the local action mapping $I= (I_1,\dots,I_n) \colon \, D\to \Delta \approx \R^n_{>0}$ (see \eqref {eq:I:})
is injective on the intersection of $D$ with a small neighborhood $V\subset\R^n$ of the point $\mathcal F(P)$. 
Moreover, $I|_{D\cap V}$ is a homeomorphism onto the intersection of the positive octant $\Delta$ with a small neighborhood of the origin.
\end{proposition}

We will give a proof of Proposition \ref {pro:inject} in Sec.~\ref {app:loc}, Step 6.

\begin{proposition} [Equivariance of the local action mapping under involutions] \label {pro:equivar}
For $n\ge2$, the swallowtail domain $D$ (see \eqref {eq:D}) is invariant under the involution $\phi^{-1}\circ\psi\circ\phi: (\R^n,0) \to (\R^n,0)$, where $\psi: (\R^n,0) \to (\R^n,0)$ is the involution having the following form in the local coordinates $\phi=(\widetilde h,\tilde\lambda)$
from Theorem \ref {thm:An:loc}:
\begin{equation} \label {eq:psi}
\psi: (\R^n,0) \to (\R^n,0), \quad 
\psi(\widetilde h,\tilde\lambda_1,\tilde\lambda_2,\dots,\tilde\lambda_{n-2},\tilde\lambda_{n-1}) = 
\left\{ \begin{array} {ll}
(\widetilde h,-\tilde\lambda_1,\tilde\lambda_2,\dots,-\tilde\lambda_{n-2},\tilde\lambda_{n-1}) & \mbox{if $n$ is odd}, \\
(-\widetilde h,\tilde\lambda_1,-\tilde\lambda_2,\dots,-\tilde\lambda_{n-2},\tilde\lambda_{n-1}) & \mbox{if $n$ is even}.
\end{array} \right.
\end{equation}
Consider the involution $\chi: \Delta\to \Delta$ having the form $(a_1,\dots,a_n)\mapsto(a_n,\dots,a_1)$, $(a_1,\dots,a_n)\in\R^n_{>0}$, under the natural identification $\Delta\approx\R^n_{>0}$. Then the local action mapping $I= (I_1,\dots,I_n) \colon \, D\to \Delta \approx \R^n_{>0}$ (see \eqref {eq:I:}) 
is equivariant under the involutions 
$\phi^{-1}\circ\psi\circ\phi: D\to D$ and $\chi: \Delta\to \Delta$, more specifically:
$\chi\circ I = I\circ(\phi^{-1}\circ\psi\circ\phi)$.
\end{proposition}

\begin{proof}
By Theorem \ref {thm:An:loc}, we have symplectic coordinates $(\tilde\lambda,\tilde\mu,\tilde p,\tilde q)$ on a small neighborhood of the point $P$.
Define (in these coordinates) a self-diffeomorphism of a (small complex) neighborhood of the point $P$ by the formula
\begin{equation} \label {eq:Psi}
\Psi\colon\,
(\tilde\lambda,\tilde\mu,\tilde p,\tilde q)\mapsto 
\left\{ \begin{array} {ll}
(-\tilde\lambda_1,\tilde\lambda_2,-\tilde\lambda_3,\dots,\tilde\lambda_{n-1},-\tilde\mu_1,\tilde\mu_2,-\tilde\mu_3,\dots,\tilde\mu_{n-1},-\tilde p,-\tilde q) & \mbox{if $n$ is odd}, \\
(\tilde\lambda_1,-\tilde\lambda_2,\tilde\lambda_3,\dots,\tilde\lambda_{n-1},-i\tilde\mu_1,i\tilde\mu_2,-i\tilde\mu_3,\dots,-i\tilde\mu_{n-1},i\tilde p,-\tilde q) & \mbox{if $n$ is even}.
\end{array} \right.
\end{equation}
This diffeomorphism, at first, is a canonical involution if $n$ is odd; futhermore it is of order $4$ with $\Psi^*\omega=-i\omega$ if $n$ is even. At second, it preserves the Liouville foliation and induces the involution \eqref {eq:psi} on $\R^n$, i.e., satisfies the relation
\begin{equation} \label {eq:Psi:psi}
\mathcal F_{can}\circ\Psi = \psi\circ \mathcal F_{can}.
\end{equation}
At third, $\Psi$ induces the following permutation of the vanishing cycles \eqref {eq:vanish}:
\begin{align*}
\Psi(\gamma_{h,\varepsilon,k}) & = \gamma_{h,\varepsilon,n-k+1} &  \mbox{ if $n$ is odd}, & \\
\Psi(\gamma_{h,\varepsilon,k}) & = \gamma_{h,\varepsilon,n-k+1}^{(-1)^{k-1}} & \mbox{ if $n$ is even}, & \qquad 
1\le k\le n
\end{align*}
(here $\gamma_{h,\varepsilon,k}^{-1}$ means the cycle obtained from the cycle $\gamma_{h,\varepsilon,k}$ by reversing orientation).
Hence, the standard swallowtail domain $\phi(D)$ is invariant under the involution \eqref {eq:psi}. If we identify the swallowtail domain with the positive octant $\Delta$ via the homeomorphism \eqref {eq:I:} near the origin (due to Proposition \ref {pro:inject}), the involution \eqref {eq:psi} induces the following involution $\chi := I\circ\phi^{-1} \circ \psi \circ\phi\circ I^{-1} \colon\, \Delta\to\Delta$ on the positive octant $\Delta \approx \R^n_{>0}$:
\begin{align*}
\chi(a_1,ia_2,a_3,ia_4,\dots,a_n) & = (a_n,\dots,ia_4,a_3,ia_2,a_1) &  \mbox{ if $n$ is odd, $\eta=+1$}, \\
\chi(ia_1,a_2,ia_3,a_4,\dots,ia_n) & = (ia_n,\dots,a_4,ia_3,a_2,ia_1) &  \mbox{ if $n$ is odd, $\eta=-1$}, \\
\chi(ia_1,a_2,ia_3,a_4,\dots,a_n) & = (ia_n,\dots,ia_4,a_3,ia_2,a_1) &  \mbox{ if $n$ is even},
\end{align*}
where $(a_1,a_2,a_3,a_4,\dots,a_n) \in \R^n_{>0}$. 
We see that the involution $\chi:\Delta\to\Delta$ has the desired form $(a_1,\dots,a_n)\mapsto (a_n,\dots,a_1)$ under the natural identification $\Delta \approx \R^n_{>0}$.
We have $\chi\circ I = I\circ(\phi^{-1}\circ\psi\circ\phi)$, as required.
\end{proof}

\section {Semi-local symplectic normal form} \label {sec:semiloc}

Suppose that $P\in M$ is unfolded $A_n$ singular point, and $\mathcal O$ is the $\R^n$-orbit containing this point. Suppose that this orbit is compact. 
Clearly, it is $(n-1)$-dimensional torus. 
It follows from \cite {Zung2000} (using \eqref{eq:Morse:type:}) that, in a neighbourhood of $\mathcal O$, there exists a Hamiltonian locally-free action of the $(n-1)$-dimensional torus $T^{n-1}$ with one of the orbits given by $\mathcal O$. This action is generated by some functions $J_1,\dots,J_{n-1}$ which are real-analytic functions of $(F_1,\dots,F_n)$. This action is in fact free (this can be shown from the local presentation \eqref {eq:Morse:type:}).

After this, we can replace the momentum map $\mathcal F=(F_1,\dots,F_n)$ with $(H,J_1,\dots,J_{n-1})$. Thus, we can further assume that $F_j=J_j=\varepsilon_j$ for $1\le j\le n-1$ in Sec.~\ref{sec:An}.

By applying the symplectic reduction procedure (see \eqref {eq:coord:loc}, \eqref {eq:Omega}),
we can introduce symplectic coordinates \\
$(J_1,\dots,J_{n-1}, \varphi_1,\dots,\varphi_{n-1},u,v)$ on a neighbourhood of the whole orbit $\mathcal O$, with angular coordinates $\varphi_j=\varphi_j\mod2\pi$, in which the 2-form $\omega$ reads \eqref{eq:Omega} in the whole neighborhood of the singular orbit $\mathcal O$, where $F_j=J_j$ and $\tau_j=\varphi_j\in\R/2\pi\Z$, $1\le j\le n-1$. After this, we can apply coordinate changes \eqref {eq:Phi1}, \eqref {eq:phi1} which bring the pair $(\omega,\mathcal F)$ to the preliminary normal form (similar to \eqref {eq:Morse:type:}) on this neighborhood 
\begin{equation*} 
(\widehat\Phi^{-1})^*\omega = \sum_{j=1}^{n-1} dJ_j \wedge d\varphi_j + d\alpha, \qquad 
\widehat\phi \circ \mathcal F \circ \widehat\Phi^{-1}
= \mathcal F_{can} := (H_{can}^{\lambda(J_1,\dots,J_{n-1})}(p,q), \lambda(J_1,\dots,J_{n-1})).
\end{equation*}

Our main semi-local result is the following normal form result and symplectic classification for 
unfolded $A_n$ singular orbits, which we will get from our local result (Theorem \ref {thm:An:loc}).

\begin{theorem} [Symplectic normal form for unfolded $A_n$ singular orbits]
\label{thm:An:semiloc}
Let $\mathcal O$ be a compact unfolded $A_n$ singular orbit of a real-analytic integrable system $\mathcal F = (F_1, \dots, F_n) \colon M^{2n} \to \mathbb R^n$, $n\ge2$. 
Consider real-analytic canonical coordinates $\Phi = (\tilde J = \tilde\lambda, \tilde \mu, \tilde p,\tilde q)$ and the functions $(\tilde H,\tilde J_1,\dots,\tilde J_{n-1}) = \phi\circ \mathcal F$ in a neighborhood of some point $P \in \mathcal O$ as in Theorem \ref {thm:An:loc}.
Let $J_j(F_1,\dots,F_n) = J_j\circ\mathcal F$, 
$1\le j\le n-1$, be the first integrals of the system on a neighbourhood of $\mathcal O$ 
generating $2 \pi$-periodic Hamiltonian flows in a neighbourhood of $\mathcal O$ such that 
$J_j(\mathcal F(\mathcal O)) = 0$ and 
$\det(\partial_{\tilde J_i}(J_j\circ\phi^{-1}))_{1\le i,j\le n-1} > 0$ at the origin (they exist, due to \cite{Zung2000}).
Then the functions $\tilde H,\tilde J_1,\dots,\tilde J_{n-1}$ and the symplectic structure $\omega$ have the following (semi-local) normal form near $\mathcal O$:
\begin{equation} \label {eq:semiloc}
\tilde H = \eta P^2 + Q^{n+1} + \sum_{j=1}^{n-1} \tilde\lambda_j Q^j, \quad \tilde J_j = \tilde\lambda_j, \qquad
\omega = \sum_{j=1}^{n-1} \textup{d} (J_j\circ\phi^{-1}(\tilde H,\tilde J_1,\dots,\tilde J_{n-1})) \wedge \textup{d} \psi_j + \textup{d}P \wedge \textup{d}Q
\end{equation}
in some real-analytic coordinates 
$(\tilde J=\tilde\lambda, \psi,P,Q) : U \to D^{n-1} \times T^{n-1}\times D^2$ 
on a neighbourhood $U$ of $\mathcal O$, in which $\mathcal O=(0,\dots,0)\times T^{n-1}\times(0,0)$. 
The tuple of functions $J\circ\phi^{-1} = (J_1\circ\phi^{-1},\dots,J_{n-1}\circ\phi^{-1})^T$ is rigid in the following sense: 
it is uniquely defined if $n=2$;
it is uniquely defined up to replacing it with $AJ\circ\phi^{-1}$, $A\in SL(n-1,\mathbb Z)$, if $n\ge4$ is even;
it is uniquely defined up to replacing it with $AJ\circ\phi^{-1}$ or $AJ\circ\phi^{-1}\circ\psi$, $A\in SL(n-1,\mathbb Z)$, if $n\ge3$ is odd (here $\psi$ is the involution \eqref {eq:psi}).
\end{theorem} 

\begin{remark} \label {rem:semiloc:}
Remark that the semi-local normal form \eqref {eq:semiloc} depends on the following parameters: 
\begin{itemize}
\item if $n=2m$ is even, then $\eta=+1$ and the parameter is the mapping germ
\begin{equation} \label {eq:J}
J\circ\phi^{-1} = (J_1,\dots,J_{n-1})^T\circ\phi^{-1} : (\R^n,0) \to (\R^{n-1},0),
\end{equation}
which is the tuple of function germs $J_j\circ\phi^{-1}$ at the origin such that the functions 
$J_j\circ\phi^{-1}(\tilde H,\tilde J_1,\dots,\tilde J_{n-1})
= J_j\circ\mathcal F$ generate a free $T^{n-1}$-action near $\mathcal O$, $1\le j\le n-1$;
\item if $n=2m+1$ is odd, then the parameters are $\eta=\pm1$ 
and the unordered pair of mapping germs $\{J\circ\phi^{-1},\ 
J\circ\phi^{-1}\circ\psi\}$ (see \eqref {eq:J} and \eqref {eq:psi}).
\end{itemize} 
Let us fix $n\ge3$ and $\eta=\pm1$.
Then the normal form \eqref {eq:semiloc} is not unique, since it depends on the choice of basic cycles on the torus $T^{n-1}$, as well as on ordering of the unordered pair $\{J\circ\phi^{-1},\ J\circ\phi^{-1}\circ\psi\}$ when $n$ is odd. 
If we change the basic cycles using an integer matrix $A=(a^i_j)\in SL(n-1,\Z)$, we obtain corresponding changes of the angular variables $\psi_j$ and their canonically conjugated variables $J_i$ in \eqref {eq:semiloc} (using the matrices $A^{-1}$ and $A^T$, respectively). 
After this change, the mapping germ $J\circ\phi^{-1}$ will be replaced with 
\begin{equation} \label {eq:J:}
J'\circ\phi^{-1}=(J_1',\dots,J_{n-1}')^T\circ\phi^{-1} : (\R^n,0) \to (\R^{n-1},0), \qquad J_i'=\sum_{j=1}^{n-1}a^i_j J_j,
\end{equation}
which is different from the initial germ \eqref {eq:J} if the matrix $A=(a^i_j)$ is not the identity matrix (this follows from non-degeneracy of the matrix $(\partial_{\tilde J_i}(J_j\circ\phi^{-1}))_{1\le i,j\le n-1}$ at the origin in Theorem \ref {thm:An:semiloc}). Thus the mapping germ \eqref {eq:J:} yields another normal form of the same integrable system (written in other coordinates).
This happens, as soon as the mapping germs \eqref{eq:J} and \eqref{eq:J:} belong to the same $SL(n-1,\Z)$-orbit under the action of the group $SL(n-1,\Z)$ on the set of mapping germs $f=(f_1,\dots,f_{n-1})^T : (\R^n,0)\to(\R^{n-1},0)$ given by
$$
Af=(\sum_{j=1}^{n-1}a^1_j f_j,\dots,\sum_{j=1}^{n-1}a^{n-1}_j f_j)^T, \quad A=(a^i_j)\in SL(n-1,\Z).
$$
\end{remark}

\begin{remark} \label {rem:semiloc}
Let us show how to choose (for $n\ge2$) a unique representative $[f]$ in the $SL(n-1,\Z)$-orbit of a germ $f=(f_1,\dots,f_{n-1})^T : (\R^n,0) \to (\R^{n-1},0)$. We will assume that the vectors 
$v_i = \frac{\partial f_i(0,\dots,0)}{\partial(\tilde\lambda_1,\dots,\tilde\lambda_{n-1})} \in \R^{n-1}$, $i=1,\dots,n-1$, are linearly independent and define a positive orientation
(this is true for the mapping germ $f := J\circ\phi^{-1}$ from Theorem \ref {thm:An:semiloc}).
If $n=2$, put $[f]:=f$.
Suppose that $n\ge3$. Denote by 
$Z\subset\R^{n-1}$ the subgroup (lattice) generated by the vectors $v_i$, $1\le i\le n-1$. Choose vectors $[v_1],\dots,[v_{n-1}] \in Z$ inductively as follows. 
Let $[v_1]\in Z$ be the shortest non-zero element of $Z$.
If $[v_1],\dots,[v_i]$ are chosen (for some $1\le i\le n-2$), let $[v_{i+1}]\in Z$ be the shortest element of $Z\setminus\langle[v_1],\dots,[v_i]\rangle$.
Clearly, the vectors $[v_1],\dots,[v_{n-1}]$ obtained in this way form a basis of $Z$. 
Consider an integer matrix $A=(a^i_j)\in SL(n-1,\Z)$ such that $[v_i] = \sum_{j=1}^{n-1}a^i_j v_j$. Define the mapping germ
$$
[f] := (\sum_{j=1}^{n-1}a^1_j f_j,\dots,\sum_{j=1}^{n-1}a^{n-1}_j f_j)^T : (\R^n,0)\to(\R^{n-1},0).
$$
Clearly, $[f]$ belongs to the $SL(n-1,\Z)$-orbit of the mapping germ $f=(f_1,\dots,f_{n-1})^T$.
It may happen that such a mapping germ $f$ is not unique; if this is the case, there are finitely many such germs, and we can choose a unique germ $[f]$ among them (the ``biggest'' one) using lexicographic order on the set of germs (more specifically, one considers such a mapping germ $f$ as a sequence of its Taylor coefficients, listed in some order, then the set of such sequences, $\R^\N$, has a usual lexicographic order).
 
A mapping germ $f$ will be called {\it special} if $f = [f]$. 
Our construction of $[f]$ implies that the $SL(n-1,\Z)$-orbit of $f$ containes a unique special mapping germ. Due to Theorem \ref {thm:An:semiloc} and 
Remark \ref {rem:semiloc:}, the mapping germs $J\circ\phi^{-1}$ and $[J\circ\phi^{-1}]$ correspond to symplecically equivalent integrable systems near their unfolded $A_n$ singular orbits. 
\end{remark}

This allows us to get the following semi-local symplectic classification for unfolded $A_n$ singularities.

\begin{corollary} [Symplectic classification of unfolded $A_n$ singular orbits] \label {cor:An:semiloc:class}
Under the hypotheses of Theorem \ref {thm:An:semiloc}, let $[J\circ\phi^{-1}]: (\R^n,0)\to(\R^{n-1},0)$ be the special mapping germ (see Remark \ref{rem:semiloc}) in the $SL(n-1,\Z)$-orbit of the mapping germ $J\circ\phi^{-1}: (\R^n,0)\to(\R^{n-1},0)$.
If $n$ is even, put $\ldbrack J\circ\phi^{-1}\rdbrack := [J\circ\phi^{-1}]$.
If $n$ is odd, consider two mapping germs $[J\circ\phi^{-1}], [J\circ\phi^{-1}\circ\psi]: (\R^n,0)\to(\R^{n-1},0)$, choose one of them which is ``not smaller'' than the other with respect to lexicographic order, and denote it by $\ldbrack J\circ\phi^{-1}\rdbrack$.
Then the real-analytic mapping germ $\ldbrack J\circ\phi^{-1}\rdbrack$ classifies the singularity at the orbit $\mathcal O$ up to real-analytic (semi-local, left-right) symplectic equivalence. 
In other words, two integrable systems $(M^i,\omega^i,\mathcal F^i)$ are symplectically equivalent near their unfolded $A_n^\eta$ singular orbits $\mathcal O^i$ ($i=1,2$, $n\ge2$) if and only if $\ldbrack J^1\circ(\phi^1)^{-1}\rdbrack = \ldbrack J^2\circ(\phi^2)^{-1}\rdbrack$.
Here $\Phi^i$ are symplectic coordinates centered at a point $P^i\in\mathcal O^i$ in which the $i$-th system has the local normal form 
$\phi^i\circ\mathcal F^i = \mathcal F^i_{can}\circ\Phi^i$ for some diffeomorphism germ $\phi^i:(\R^n,\mathcal F^i(\mathcal O^i))\to (\R^n,0)$.
\end {corollary}

We will prove Theorem \ref {thm:An:semiloc} and Corollary \ref {cor:An:semiloc:class} in Sec.~\ref {app:semiloc}.

Having the above definitions of the (semi-local) action variables $J_j\circ\mathcal F$, $1\le j\le n-1$, and the local action variables $I_k\circ\mathcal F$, $1\le k\le n$, see \eqref {eq:I:},
in mind, we can state the following result.

\begin{corollary} [Completeness of action variables for unfolded $A_n$ singular orbit] \label{cor:An:semiloc}
Consider a pair of $n$-degree of freedom real-analytic integrable systems $(M^i,\omega^i,\mathcal F^i)$, $i = 1, 2,$
having unfolded $A_n^\eta$ singularity at compact orbits $\mathcal O^i$ ($n\ge2$, $\eta=\pm1$). Let 
$V_i \simeq D^{n+1} \times T^{n-1}$ be a sufficiently small neighbourhood of $\mathcal O^i$. 
Consider the functions $J^i_j$, $1\le j\le n-1$, on the swallowtail domain $D_i\subset\R^n$ (see \eqref {eq:D}) such that the functions $J^i_j\circ\mathcal F^i$ generate a Hamiltonian free $T^{n-1}$-action on $V_i$ and 
$\det(\partial_{\tilde J^i_i}(J^i_j\circ(\phi^i)^{-1}))_{1\le i,j\le n-1} > 0$ at the origin 
 ($i=1,2$)\footnote{Remark that Corollary \ref {cor:An:semiloc} is about the functions $J^i_j$ of $(F_1^i,\dots,F_n^i)$, in contrast to Theorem \ref {thm:An:loc} which is about the functions $J^i_j\circ(\phi^i)^{-1}$ of $(\tilde H^i,\tilde J_1^i,\dots,\tilde J_{n-1}^i)$.}.
Put $m:=[\frac n2]$ and consider 
the function $I^i_{m+1}$ on the swallowtail domain $D_i$ such that 
the function $I^i_{m+1}\circ\mathcal F^i$ is the local action variable corresponding to the family of ``central'' vanishing cycles $\gamma^i_{h,\varepsilon,m+1}$ of the $i$-th system, $(h,\varepsilon)\in D_i$.
Then $\mathcal F^1$ and $\mathcal F^2$ are symplectically equivalent near the singular orbits $\mathcal O^1$ and $\mathcal O^2$ if and only if there exists a diffeomorphism germ $\phi \colon (\mathbb R^n,\mathcal F^1(\mathcal O^1)) \to (\mathbb R^n,\mathcal F^2(\mathcal O^2))$ 
that respects the swallowtail domains, $\phi(D_1)=D_2$, and makes the action variables equal, 
$\ldbrack J^1\rdbrack = \ldbrack J^2 \rdbrack \circ \phi \mbox{ and } 
I^1_{m+1} = I^2_{m+1} \circ \phi$ on $D_1$.
Here $\ldbrack J^i\circ(\phi^i)^{-1}\rdbrack:(\R^n,0)\to(\R^{n-1},0)$ is the special mapping germ (see Remark \ref{rem:semiloc}) defined unambiguously as in Corollary \ref {cor:An:semiloc:class},
$\ldbrack J^i\rdbrack := \ldbrack J^i\circ(\phi^i)^{-1}\rdbrack \circ \phi^i$ ($i=1,2$).
\end{corollary}

\begin{proof}
Let us prove the ``only if'' part.
Suppose that two integrable systems $(M^i,\omega^i,\mathcal F^i)$ are symplectically equivalent 
on some neighborhoods $U^i$ of their unfolded $A_n^\eta$ singular orbits $\mathcal O^i$ ($i=1,2$).
This means that there exist a symplectomorphism $\Phi:U^1\to U^2$ and a diffeomorphism germ $\phi:(\R^n,\mathcal F^1(\mathcal O^1))\to (\R^n,\mathcal F^2(\mathcal O^2))$ such that $\phi\circ\mathcal F^1 = \mathcal F^2 \circ \Phi$.
Since $\Phi$ is a fibrewise homeomorphism, $\phi$ respects the swallowtail domains, $\phi(D_1)=D_2$ (see Sec.~\ref {app:loc}, Step 5, for a proof). 
Finally, $\Phi$ transforms vanishing cycles of the first system to those of the second system and either preserves or reverses numbering of the vanishing cycles $\gamma^i_{h,\varepsilon,k}$, $1\le k\le n$, $(h,\varepsilon)\in D_i$ (see Sec.~\ref {app:loc}, Step 6, for a proof), moreover reversion of numbering is possible for odd $n$ only. 
We can and will assume that $\Phi$ preserves numbering of vanishing cycles 
(otherwise $n$ is odd and we replace $\Phi$ and $\phi$ with 
$((\Phi^2)^{-1}\circ\Psi\circ\Phi^2)\circ\Phi$ and $(\phi_2^{-1}\circ\psi\circ\phi_2)\circ\phi$, respectively; this replacement preserves the ``central'' vanishing cycle).
As follows from the proof of Corollary \ref{cor:An:semiloc:class}, the $SL(n-1,\Z)$-orbits of the mapping germs $J^1$ and $J^2\circ\phi$ coincide, thus the special elements (see Remark \ref{rem:semiloc}) of these orbits coincide too, i.e., $[J^1]=[J^2\circ\phi]$. 
Since $\Phi$ is a fibrewise symplectomorphism preserving numbering of vanishing cycles, it preserves the action variables corresponding to vanishing cycles, hence 
$I^1 = I^2 \circ \phi$ (since 
$I^1\circ\mathcal F^1 = I^2\circ\mathcal F^2\circ\Phi = I^2\circ\phi\circ\mathcal F^1$). 
Thus, $\phi$ transforms action variables of the second system to the action variables of the first system.

Let us prove the ``if'' part.
Suppose that two integrable systems $(M^i,\omega^i,\mathcal F^i)$ are given near their unfolded $A_n^\eta$ singular orbits $\mathcal O^i$ ($i=1,2$). Suppose that there exists a diffeomorphism germ $\phi \colon (\mathbb R^n,\mathcal F^1(\mathcal O^1)) \to (\mathbb R^n,\mathcal F^2(\mathcal O^2))$ that respects the swallowtail domains, $\phi(D_1)=D_2$, and makes the action variables equal, $I^1_{m+1} = I^2_{m+1} \circ \phi$ and $[J^1] = [J^2 \circ \phi]$ on $D_1$.
By changing a basis on the torus $\mathcal O^1$, we can replace the mapping germ $J^1$ with any element of its $SL(n-1,\Z)$-orbit, hence we can achieve that 
\begin{equation} \label {eq:J::}
J^1 = J^2 \circ \phi \qquad \mbox{on $D_1$}. 
\end{equation}
Since $\phi$ is a real-analytic diffeomorphism germ that respects the swallowtail domains and preserves the local action variable corresponding to the ``central''
vanishing cycle ($I^1_{m+1} = I^2_{m+1} \circ \phi$), it follows that it preserves the local action variable corresponding to any vanishing cycle that belongs to the orbit of the ``central'' vanishing cycle under the action of the monodromy group.
It is an easy exercise that the monodromy group transtively acts on the set of vanishing cycles. Thus we have
$$
(I_1^1,\dots,I_n^1) = (I_1^2,\dots,I_n^2)\circ\phi 
\quad \mbox{or} \quad 
(I_1^1,\dots,I_n^1) = (I_n^2,\dots,I_1^2)\circ\phi 
\qquad \mbox{on $D_1$},
$$
furthermore the latter case is possible for odd $n$ only, due to \eqref {eq:I:}. On the other hand, Theorem \ref {thm:An:loc} implies that 
$(I_1^1,\dots,I_n^1) = (I_1^2,\dots,I_n^2)\circ(\phi^2)^{-1}\circ\phi^1$ on $D_1$.
Since the local action mapping $(I_1^2,\dots,I_n^2)\colon\, D_2\to \Delta \approx \R^n_{>0}$ is injective by Proposition \ref {pro:inject}, it follows that
\begin{equation} \label {eq:tilde:HJ}
\phi^1 = \phi^2 \circ \phi  
\quad \mbox{or}  \quad 
\phi^1 = \psi \circ \phi^2 \circ \phi
\qquad \mbox{on $D_1$}, 
\end{equation}
respectively.
It follows from the equalities \eqref {eq:J::} and \eqref {eq:tilde:HJ} that they hold on a whole neighborhood of the origin in $\R^n$ by uniqueness of analyic continuation, furthermore the mapping germ $J^1\circ(\phi^1)^{-1}$ coincides with the mapping germ $J^2\circ(\phi^2)^{-1}$ or $J^2\circ(\phi^2)^{-1}\circ\psi$, respectively. Therefore, both integrable systems have the same canonical form \eqref{eq:semiloc} near their singular orbits $\mathcal O^i$, up to replacements from the rigidity assertion in Theorem \ref {thm:An:semiloc}. 
Hence, due to Theorem \ref {thm:An:semiloc}, Remark \ref {rem:semiloc:} and \eqref {eq:Psi:psi}, the systems are symplectically equivalent near these orbits.
\end{proof}

Note that Theorem~\ref{thm:An:semiloc} implies Corollary~\ref{cor:An:semiloc} (this corollary for $n=2$ was already proven in \cite{Bolsinov2018},
so Theorem~\ref{thm:An:semiloc} can be seen as a generalization of the result about completeness of the action variables).

\section {Semi-global symplectic classification} \label {sec:semiglob}

Sonsider the integral surface $\mathcal L = \{\mathcal F=\const\}$ containing a singular orbit $\mathcal O$ having unfolded $A_n$ singularity type. Suppose that $\mathcal L$ is connected and compact, and 
all points of $\mathcal L\setminus\mathcal O$ are regular.

Denote by $\ell$ the number of connected components of $\mathcal L\setminus\mathcal O$, i.e., $\ell := |\pi_0(\mathcal L\setminus\mathcal O)|$. Let us show that 
\begin{equation} \label {eq:ell}
\ell=1 \quad \mbox{if $n$ is even}, \qquad \qquad 
\ell=1-\eta \quad \mbox{if $n$ is odd}.
\end{equation} 
Consider the reduced Hamiltonian $H^{\tilde\lambda}_{can}(P,Q) = \eta P^2 + Q^{n+1} + \sum_{j=1}^{n-1}\tilde\lambda_jQ^j$, so $H^0_{can}(P,Q) = \eta P^2 + Q^{n+1}$ is $A_n$ singularity. On a small neighbourhood of the origin in the plane $(P,Q)$, the curve $\{H^0_{can}(P,Q)=0\}$ is a graph, with a vertex of degree $2\ell$ at the origin, where $\ell$ is given by \eqref {eq:ell}, as can be seen from Fig.~\ref {fig:A123}.

\begin{figure}[htbp]
\begin{center}
\includegraphics[width=0.6\linewidth]{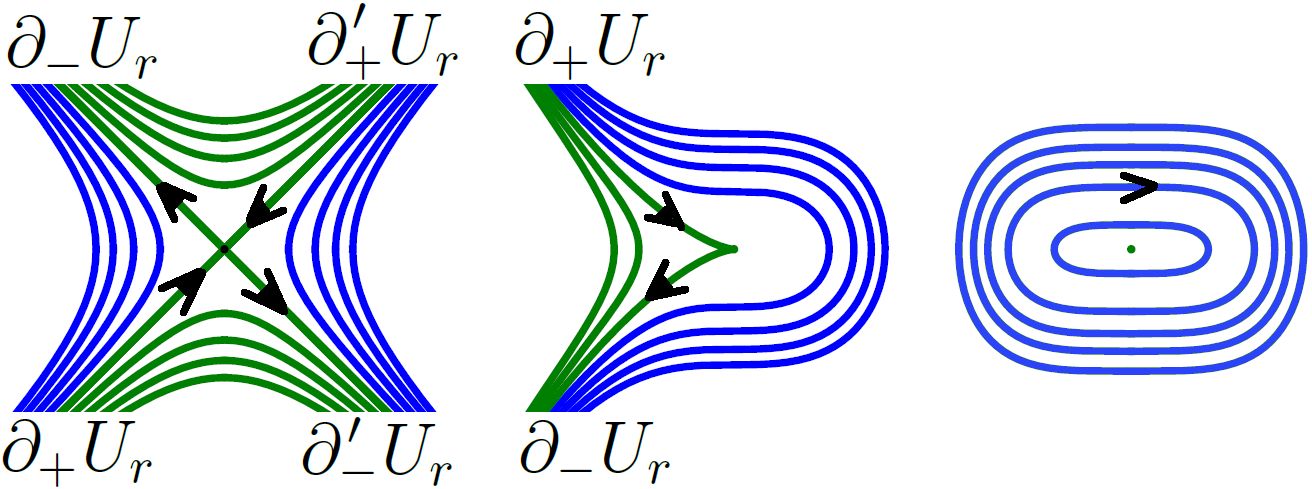} \quad
\includegraphics[width=0.25\linewidth]{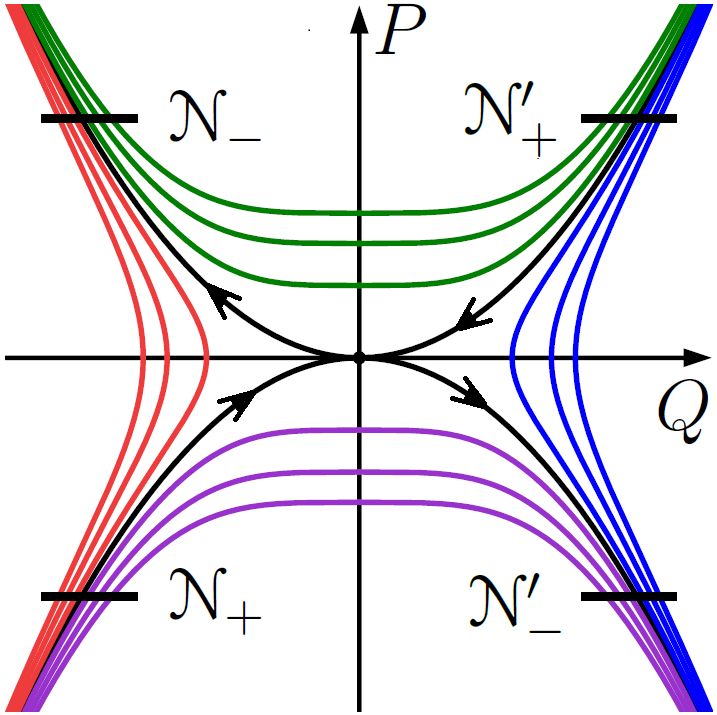}
\end{center}
\caption{Level sets of $A_n$ singularity: $n=1$, $\eta=-1$ (left), $n=2m$ (middle left), $n=2m+1$ with $\eta=+1$ (middle right) or $\eta=-1$ (right). Positive and negative boundaries $\partial_\pm U_r$, $\partial_\pm'U_r$ of $U_r$ (left) and the Lagrangian sections $\mathcal N_\pm\subset\partial_\pm U_r$, $\mathcal N_\pm'\subset\partial_\pm' U_r$ (right) shown in the plane $(P,Q)$}
\label{fig:A123}
\end{figure}

Before formulating our main result, let us construct semi-global invariants of the singularity. 
Recall that we already defined a local topological invariant (which is a sign $\eta=\pm1$ for odd $n$, see Theorem \ref {thm:An:loc}) and a semi-local symplectic invariant (which is a special mapping germ $\ldbrack J\circ\phi^{-1}\rdbrack:(\R^n,0)\to (\R^{n-1})$, see 
Corollary \ref {cor:An:semiloc:class}). 
Below (in Steps 1--3), we will construct two semi-global invariants; namely, a permutation $\beta\in \Sym_\ell$ (a topological invariant) and a mapping germ $S:(\R^n,0)\to (\R^\ell,0)$ (a symplectic invariant).

Suppose that $\mathcal O=\mathcal L$, i.e., $\ell=0$. This means that $n$ is odd and $\eta=+1$. In this case, the semi-local classification (Theorem \ref {thm:An:semiloc} and Corollaries \ref{cor:An:semiloc:class} and \ref{cor:An:semiloc}) works in semi-global setting as well, so we are done.

Suppose that $\mathcal L\ne\mathcal O$, i.e., $\ell>0$. In according to \eqref {eq:ell}, 
this means that either $n$ is odd and $\eta=-1$, or $n$ is even (thus $\eta=+1$). We will construct semi-global invariants in three steps.

Step 1. 
Consider the Hamiltonian $T^{n-1}$-action near the orbit $\mathcal O$ generated by the first integrals $J_j\circ\mathcal F$, $1\le j\le n-1$, from Theorem \ref {thm:An:semiloc}. As was proved by N.T.~Zung, these functions generate a Hamiltonian $T^{n-1}$-action near the fiber $\mathcal L$.

By Theorem \ref {thm:An:semiloc}, there exist real-analytic diffeomorphism germ $\phi:(\R^n,\mathcal F(\mathcal O)) \to (\R^n,0)$ and coordinates
\begin{equation} \label {eq:coord}
\Phi_U = (\tilde \lambda_1,\dots,\tilde \lambda_{n-1},\psi_1,\dots,\psi_{n-1},P,Q) : U \to D^{n-1}\times T^{n-1}\times D^2
\end{equation}
on a neighborhood $U$ of the singular orbit $\mathcal O$ in which the presentation \eqref {eq:semiloc} holds, moreover these coordinates can and will be chosen in such a way that 
$J\circ\phi^{-1} = \ldbrack J\circ\phi^{-1}\rdbrack$ (see Corollary \ref {cor:An:semiloc:class}).
We will also suppose that the image of $U$ under the coordinate map \eqref {eq:coord} is $T^{n-1}$-invariant and contains the (singular) curve $\{\tilde \lambda_1=\dots=\tilde \lambda_{n-1}=0,\ \psi=0,\ P=\pm t^{n+1},\ Q=\pm t^2 \mid t\in[0,1]\}$
(the more general situation will be considered in Step 3 below, using a small number $t_0\in(0,1]$). 

Note that the coordinates \eqref {eq:coord} are rigid in the following sense: they are uniquely defined up to replacing $\Phi_U$ with $\phi_{S_U\circ\mathcal F}^1\circ\Phi_U$ (or with $\phi_{S_U}^1\circ\widetilde\Psi\circ\Phi_U$ when either $n=1$ or $n\ge3$ is odd and $[J\circ\phi^{-1}] = [J\circ\phi^{-1}\circ\psi]$) where $\phi_{S_U\circ\mathcal F}^1$ is the time-one map of the Hamiltonian flow generated by a real-analytic function $S_U(F_1,\dots,F_n)=S_U\circ\mathcal F$, $\widetilde\Psi(\tilde \lambda_1,\dots,\tilde \lambda_{n-1},\psi_1,\dots,\psi_{n-1},P,Q) := (-\tilde \lambda_1,\tilde \lambda_2,-\tilde \lambda_3,\dots,\tilde \lambda_{n-1},-\psi_1,\psi_2,-\psi_3,\dots,\psi_{n-1},-P,-Q)$. 
Indeed: the diffeomorphism germ $\phi$ is that from Theorem \ref {thm:An:loc}, hence it is uniquely defined when $n$ is even or $[J\circ\phi^{-1}] \ne [J\circ\phi^{-1}\circ\psi]$ (respectively, it is uniquely defined up to replacing $\phi$ with $\psi\circ\phi$ when $n$ is odd and $[J\circ\phi^{-1}] = [J\circ\phi^{-1}\circ\psi]$), due to rigidity assertion in Theorem \ref {thm:An:loc} and the assumption $J\circ\phi^{-1} = \ldbrack J\circ\phi^{-1}\rdbrack$. 
Hence if $\widetilde\phi,\widetilde\Phi_U$ is another pair of diffeomorphism germs satisfying \eqref {eq:semiloc} and $J\circ\widetilde\phi^{-1} = \ldbrack J\circ\widetilde\phi^{-1}\rdbrack$, then $\Phi_1 := \widetilde\Phi_U\circ\Phi_U^{-1}$ is an $\mathcal F_{can}$-preserving symplectomorphism, $\mathcal F_{can}\circ\Phi_1 = \mathcal F_{can}$ 
(or $\mathcal F_{can}\circ\Phi_1 = \psi\circ\mathcal F_{can} = \mathcal F_{can}\circ\widetilde\Psi$, respectively). 
Since our singular fiber $\mathcal L$ contains a regular point, this implies that $\Phi_1$ has the form $\Phi_1 = \phi_{S_U}^1$, i.e., 
$\widetilde\Phi_U = \phi_{S_U}^1\circ\Phi_U$ 
(respectively, $\Phi_1\circ\widetilde\Psi = \phi_{S_U}^1$, 
thus
$\widetilde\Phi_U = \phi_{S_U}^1\circ\widetilde\Psi\circ\Phi_U$), as required.

Fix a small real number $r>0$ and consider the closed neighborhood $U_r\subset U$ of $\mathcal O$ having the form
$$
U_r := U\cap \{ 
H_{can}^{\tilde\lambda}(P, Q)^2 + |\tilde\lambda|^2 \le r^2,\ \psi\in T^{n-1},\ |P|\le t_0^{n+1} \}
$$
where $t_0=1$ (in fact, we will consider more general situation in Step 3 below).
The intersection of $U_r$ with the plane $\tilde\lambda=0$ is shown in Fig.~\ref {fig:A123}.
The following subsets of $\partial U_r$ will be called the (positive and negative) {\it boundaries} of $U_r$:
$$
\partial_\sigma U_r := U_r\cap \{ \eta\sigma P = t_0^{n+1},\ Q < 0 \}, \qquad \sigma\in\{\pm\},
$$
$$
\partial_\sigma' U_r := U_r\cap \{ \sigma P = t_0^{n+1},\ Q > 0 \}, \qquad \sigma\in\{\pm\}
$$
(note that $\partial_\sigma' U_r \ne \varnothing$ if and only if $n$ is odd).
Note that the Hamiltonian vector field $X_{\tilde H}$ is directed out of $U_r$ (respectively, in $U_r$) on $\partial_- U_r$ and $\partial_-' U_r$ (respectively, on $\partial_+ U_r$ and $\partial_+' U_r$).

Our purpose is to construct semi-global symplectic invariants of the singularity. For this, let us cut the regular neighborhood 
$\{H_{can}^{\tilde\lambda}(P, Q)^2 + |\tilde\lambda|^2 \le r^2\}$ of the fiber $\mathcal L$ into the following pieces:
$U_r$ and 
\begin{equation} \label {eq:Vr}
V_r := \overline{\{ \tilde H^2 + \sum_{j=1}^{n-1} \tilde J_j^2 \le r^2\} \setminus U_r}.
\end{equation}
Clearly, $U_r$ is well-defined (i.e., two different choices for such $U_r$ yield pieces that are related to each other by $\mathcal F$-preserving symplectomorphism).
We see that $V_r$ consists of $\ell$ connected components. For each connected component $V_r^i$ of $V_r$, we will construct a function germ $S_i\circ\phi^{-1}:(\R^n,0) \to (\R,0)$ 
that characterizes the piece $V_r^i$ up to $\mathcal F$-preserving symplectomorphism (the corresponding functions $S_i\circ\mathcal F$ are {\em generating functions}).

Let us proceed with constructing such generating functions. Consider two
Lagrangian submanifolds $\mathcal N_\pm\subset U$ having the form 
$$
\mathcal N_\sigma := \{\psi=0,\ \eta \sigma P=t_0^{n+1},\ Q < 0\}, \qquad \sigma=\pm1.
$$
If $n$ is odd (thus $\eta=-1$, $\ell=1-\eta=2$), we consider 
two more Lagrangian submanifolds $\mathcal N_\pm'\subset U$ having the form 
$$
\mathcal N_\sigma' := \{\psi=0,\ \sigma P=t_0^{n+1},\ Q > 0\}, \qquad 
\sigma=\pm1.
$$
Each of these Lagrangian submanifolds transversally intersects the Lagrangian submanifold $\mathcal L\setminus\mathcal O$ at a unique point (so, it is a local {\em Lagrangian section} of our singular Lagrangian foliation). Clearly, $\mathcal N_\sigma \subset \partial_\sigma U_r$ and 
$\mathcal N_\sigma' \subset \partial_\sigma' U_r$ ($\sigma\in \{\pm\}$). See Fig.~\ref {fig:A123}.

It follows from the rigidity property of the coordinates \eqref {eq:coord} (see above) that the union of the Lagrangian sections, $\mathcal N_+\cup\mathcal N_-$ for even $n$, $\mathcal N_+\cup\mathcal N_-\cup\mathcal N_+'\cup\mathcal N_-'$ for odd $n$, is rigid in the following sense: it is uniquely defined up to ``shifting'' it with the time-one map $\phi_{S_U\circ\mathcal F}^1$ of some $\mathcal F$-preserving Hamiltonian flow.

Step 2. 
Observe that $\overline{\mathcal L\setminus U_r}$ is $n$-dimensional manifold with boundary, moreover $\partial(\mathcal L\setminus U_r)$ is the union of several $T^{n-1}$-orbits (twice as many as the number $\ell$ in \eqref {eq:ell}). We have two cases.

{\em Case 1:} $\mathcal L\setminus\mathcal O$ is connected (i.e., $\ell=1$, $n$ is even, $\eta=+1$, a supercritical singularity by Definition~\ref {def:pleat}). 
Then the boundary of $\mathcal L\setminus U_r$ consists of two $T^{n-1}$-orbits:
$\partial(\mathcal L\setminus U_r) = \mathcal O_+ \cup \mathcal O_-$ where 
$\mathcal O_\sigma := \mathcal L \cap (\partial_\sigma U_r)$, and the Lagrangian submanifold $\mathcal L\setminus\mathcal O$ is a cylinder $T^{n-1}\times\R$. Hence, the closed neighborhood \eqref {eq:Vr} of $\mathcal L\setminus U$ is connected. 
One can show that we can move from $\mathcal N_-$ to $\mathcal N_+$ in $V_r$ via the time-one map of the Hamiltonian flow generated by a function $S(F_1,\dots,F_n) = S\circ\mathcal F$, for some real-analytic function germ $S:(\R^n,\mathcal F(\mathcal O))\to(\R,0)$.

Note that such a function $S$ is rigid in the following sense: it does not depend on the choice of the Lagrangian sections $\mathcal N_\pm$, due to rigidity property of these sections (see Step 1). On the other hand, $S$ is not unique, since it can be replaced with any function 
yielding the same time-one map. Thus, it can be replaced with any function of the form $S_a = S + 2\pi \sum_{j=1}^{n-1} a_j J_j$ where $a=(a_1,\dots,a_{n-1})\in \Z^{n-1}$.
Let us choose (a unique) $a\in\Z^{n-1}$ such that
$$
\frac{\partial (S_a\circ\phi^{-1})}{\partial(\widetilde J_1,\dots,\widetilde J_{n-1})} \in 
\left\{ \sum_{j=1}^{n-1} t_j
\frac{\partial 
(J_j\circ\phi^{-1})}{\partial(\widetilde J_1,\dots,\widetilde J_{n-1})} \mid 0\le t_j <1, \ 1\le j\le n-1\right\} \qquad \mbox{at the origin}
$$
(this can be done, since the matrix $(\partial_{\widetilde J_i} (J_j\circ\phi^{-1}))_{1\le i,j\le n-1}$ is non-degenerate at the origin).
Let us replace $S$ with $S_a$, then we will obtain a unique function germ $S$ such that
\begin{equation} \label {eq:S}
\frac{\partial (S\circ\phi^{-1})}{\partial(\widetilde J_1,\dots,\widetilde J_{n-1})} \in 
\left\{ \sum_{j=1}^{n-1} t_j
\frac{\partial (J_j\circ\phi^{-1})}{\partial(\widetilde J_1,\dots,\widetilde J_{n-1})} \mid 0\le t_j <1, \ 1\le j\le n-1
\right\} \qquad \mbox{at the origin}.
\end{equation}

{\em Case 2:} $\mathcal L\setminus\mathcal O$ is disconnected (i.e., $\ell=2$, $n$ is odd, $\eta=-1$, a subcritical singularity by Definition~\ref {def:pleat}). 
Then the boundary of $\mathcal L\setminus U_r$ consists of four $T^{n-1}$-orbits:
$\partial(\mathcal L\setminus U_r) = \mathcal O_+ \cup \mathcal O_- \cup \mathcal O_+' \cup \mathcal O_-'$ where 
$\mathcal O_\sigma := \mathcal L \cap (\partial_\sigma U_r)$, 
$\mathcal O_\sigma' := \mathcal L \cap (\partial_\sigma' U_r)$, $\sigma\in\{\pm\}$, 
and the Lagrangian submanifold $\mathcal L\setminus\mathcal O$ consists of two connected components (each of which is a cylinder $T^{n-1}\times\R$). Hence, the closed neighborhood $V_r$ of $\mathcal L\setminus U$ (see \eqref {eq:Vr}) is disconnected too and consists of two connected components, denoted by $V_r^1$ and $V_r^2$. Two subcases are possible.

{\em Subcase 1:}
$\partial V_r^1 = \partial_- U_r \cup \partial_+ U_r$ and 
$\partial V_r^2 = \partial_-' U_r \cup \partial_+' U_r$. 
Then define the identity permutation $\beta=\id \in \Sym_2$.
Similarly to Case 1, 
one can move from $\mathcal N_-$ to $\mathcal N_+$ in $V_r^1$ via the time-one map of the Hamiltonian flow generated by a function 
$S_1(F_1,\dots,F_n) = S_1\circ\mathcal F$.
Furthermore, one can move from $\mathcal N_-'$ to $\mathcal N_+'$ in $V_r^2$ via the time-one map of the Hamiltonian flow generated by a function
$S_2(F_1,\dots,F_n) = S_2\circ\mathcal F$.

{\em Subcase 2:}
$\partial V_r^1 = \partial_- U_r \cup \partial_+' U_r$ and 
$\partial V_r^2 = \partial_-' U_r \cup \partial_+ U_r$. 
Then define the non-identity permutation $\beta=(12) \in \Sym_2$.
Similarly to above, 
one can move from $\mathcal N_-$ to $\mathcal N_+'$ in $V_r^1$ via the time-one map of the Hamiltonian flow generated by a function 
$S_1(F_1,\dots,F_n) = S_1\circ\mathcal F$.
Furthermore, one can move from $\mathcal N_-'$ to $\mathcal N_+$ in $V_r^2$ via the time-one map of the Hamiltonian flow generated by a function 
$S_1(F_1,\dots,F_n) = S_1\circ\mathcal F$.

In each subcase, we have two generating functions $S_i(F_1,\dots,F_n)$, $i=1,2$.
As in Case 1, each function $S_i$ is rigid (up to replacing $(S_1\circ\phi^{-1},S_2\circ\phi^{-1})$ with $(S_2\circ\phi^{-1}\circ\psi, S_1\circ\phi^{-1}\circ\psi)$ when $n$ is odd and $[J\circ\phi^{-1}]=[J\circ\phi^{-1}\circ\psi]$), furthermore it is not unique, since it can be replaced with any function of the form $S_{i,a} = S_i + 2\pi \sum_{j=1}^{n-1} a_j J_j$ where $a=(a_1,\dots,a_{n-1})\in\Z^{n-1}$.
We again choose an appropriate $a=a_i$ and replace $S_i$ with $S_{i,a_i}$ satisfying \eqref {eq:S}.

Step 3.
The above construction of the generating function (resp., pair of generating functions) was performed under the assumption $t_0=1$. 
If $0<t_0<1$, we can make a similar construction; denote by $R_{t_0}$ the resulting generating function if $n$ is even (if $n$ is odd, we will do the same for every function of the pair).
Define $S := R_{t_0}-2T_{t_0}$ where $T_{t_0}$ denotes the generating function of a symplectomorphism that sends the Lagrangian submanifold 
$\mathcal N_{t_0,-} := \{\psi=0,\ -\eta P = t_0^{n+1},\ Q<0\}$ to $\mathcal N_{1,-}$. 
Then we can replace $S$ with $S_a$, $a\in\Z^{n-1}$, that satisfies \eqref {eq:S}.
One easily checks that such a function $S$ does not depend on the choice of $t_0$.

\begin{remark} \label {rem:semiglob}
For any $n\ge1$, we constructed $\ell$ generating functions $S_i\circ\mathcal F=S_i(F_1,\dots,F_n)$, $1\le i\le \ell$, satisfying \eqref {eq:S}, and a permutation $\beta\in\Sym_\ell$, where $\ell$ is given by \eqref {eq:ell}. More specifically:

(A) For any even $n\ge2$ (and $\eta=+1$), we have $\ell=1$ and constructed a unique generating function $S\circ\mathcal F=S(F_1,\dots,F_n)$ satisfying \eqref {eq:S}. Then the germ $S\circ\phi^{-1}$ is a symplectic invariant.

(B) For any odd $n\ge1$ (and $\eta=-1$), we have $\ell=2$ and two subcases, Subcase 1 and Subcase 2, which are topologically different if $n\ge3$ (and they are distinguished by the topological invariant $\beta\in\Sym_\ell$, a permutation showing a correspondence between positive and negative boundaries of $U_r$ via $V_r$). In each subcase, we constructed two generating functions $S_i\circ\mathcal F = S_i(F_1,\dots,F_n)$ satisfying \eqref {eq:S}.
If $n=1$ (and $\eta=-1$), we get a non-degenerate saddle singular point, and it is enough to consider just one subcase, say Subcase 1 (this can be achieved by replacing $\tilde H\to-\tilde H$); define the permutation $\beta$ to be the identity. If $n=1$ or $[J\circ\phi^{-1}]=[J\circ\phi^{-1}\circ\psi]$, then the pairs of functions $(S_1\circ\phi^{-1},S_2\circ\phi^{-1})$ and $(S_2\circ\phi^{-1}\circ\psi, S_1\circ\phi^{-1}\circ\psi)$ correspond to symplectically equivalent semi-global singularities, thus we can and will assume that the germ $S_1\circ\phi^{-1}$ is ``not smaller'' than $\tilde S_2\circ\phi^{-1}\circ\psi$ with respect to lexicographic order, 
where $\tilde S_2 \in \{S_2 + 2\pi\sum_{j=1}^{n-1}a_jJ_j\mid a=(a_1,\dots,a_{n-1})\in\Z^{n-1}\}$ 
satisfies \eqref {eq:S}. 
One easily shows that the mapping germ 
$S\circ\phi^{-1} := (S_1\circ\phi^{-1},S_2\circ\phi^{-1})$ 
is defined unambiguously, and it is a symplectic invariant for any odd $n\ge1$.
\end{remark}

\begin{theorem} [Symplectic classification of unfolded $A_n$ singular fibers] \label{thm:An:semiglob}
Suppose that the integral surface $\mathcal L = \{\mathcal F=\const\}$ containing an unfolded $A_n$ singular orbit $\mathcal O$ is connected and compact, and $\mathcal L\setminus\mathcal O$ consists of regular points of the integrable system $\mathcal F$.
Consider real-analytic diffeomorphism germ $\phi:(\R^n,\mathcal F(\mathcal O)) \to (\R^n,0)$ and real-analytic coordinates \eqref {eq:coord} on a neighborhood $U$ of $\mathcal O$ in which the presentation \eqref {eq:semiloc} holds and $J\circ\phi^{-1} = \ldbrack J\circ\phi^{-1}\rdbrack$ (see Corollary \ref {cor:An:semiloc:class}). Define the permutation $\beta\in\Sym_\ell$ and the mapping germ $S=(S_1,\dots,S_\ell):(\R^n,\mathcal F(\mathcal O)) \to (\R^\ell,0)$ unambiguously as in Remark \ref {rem:semiglob} (here $\ell=|\pi_0(\mathcal L\setminus\mathcal O)|$ is given by \eqref {eq:ell}, $S_i\circ \mathcal F$, $1\le i\le \ell$, are generating functions satisfying \eqref {eq:S}). Then the mapping germs
$J\circ\phi^{-1} : (\R^n,0)\to(\R^{n-1},0)$,
$S\circ\phi^{-1} : (\R^n,0)\to(\R^\ell,0)$ and the permutation $\beta\in\Sym_\ell$
classify the singularity at the fiber $\mathcal L$ up to real-analytic 
symplectic equivalence. In other words, two integrable systems $(M^i,\omega^i,\mathcal F^i)$ are symplectically equivalent near their unfolded $A_n^\eta$ singular fibers $\mathcal L^i$ ($i=1,2$) if and only if $J^1\circ(\phi^1)^{-1} = J^2\circ(\phi^2)^{-1}$, $S^1\circ(\phi^1)^{-1} = S^2\circ(\phi^2)^{-1}$ and $\beta^1=\beta^2$.
\end{theorem}

\begin{proof}
Step 4. Suppose that $\Phi$ and $\phi$ are chosen as in \eqref {eq:coord}.
By Corollary \ref {cor:An:semiloc:class}, the function germ $J\circ\phi^{-1} = \ldbrack J\circ\phi^{-1}\rdbrack$ is a symplectic invariant of the singular orbit $\mathcal O$.
As we showed in Remark~\ref {rem:semiglob}, the permutation $\beta\in\Sym_\ell$ and the function germ $S\circ\phi^{-1} = (S_1,\dots,S_\ell)\circ\phi^{-1}$ are symplectic invariants of the singular fiber $\mathcal L$.

To show that the function germs $\ldbrack J\circ\phi^{-1}\rdbrack$, $S\circ\phi^{-1}$ and the permutation $\beta$ are the only invariants, and they admit arbitrary values, we now construct a `model' real-analytic integrable system with an unfolded $A_n$ singular fiber from such a pair of germs and a permutation. Consider Cases 1 and 2 as in Step 2.

{\em Case 1:} $\mathcal L\setminus\mathcal O$ is connected (i.e., $\ell=1$), thus $n$ is even. 
Introduce canonical coordinates 
$$
\Phi_V = (\tilde h,\tilde\tau, \tilde \lambda_1,\dots,\tilde \lambda_{n-1},\tilde\psi_1,\dots,\tilde\psi_{n-1}) : V \to D^2 \times D^{n-1}\times T^{n-1}
$$
in which 
$(\tilde H,\tilde J_1,\dots,\tilde J_{n-1}) = (\tilde h, \tilde \lambda_1,\dots,\tilde \lambda_{n-1})$ (so, our Lagrangian foliation on $V_r$ becomes standard) and the ``negative'' Lagrangian section $\mathcal N_-$ is given by $\tilde\tau=\tilde\psi_1=\dots=\tilde\psi_{n-1}=0$. 
Observe that the closed tubular neighborhood 
$\{ \tilde H^2 + \sum_{j=1}^{n-1} \tilde J_j^2 \le r^2\}$ of $\mathcal L$ can be obtained by gluing two pieces, $\Phi_U(U_r)$ and
$\Phi_V(V_r)$, along the positive and negative boundaries $\Phi_U(\partial_\pm U_r)$ via the gluing mapping $\Phi_V\circ \Phi_U^{-1}|_{\Phi_U(\partial_\pm U_r)}$ that preserves $\mathcal F$. 

By construction of $\Phi_U$ and $U_r$ (see Step 1), the piece $\Phi_U(U_r)$ and the symplectic structure on it are uniquely determined by the mapping germ $J\circ\phi^{-1}$.
We will now show that the piece $\Phi_V(V_r)$ and the gluing mapping are uniquely determined by the generating function $S\circ\phi^{-1}$. For this, we observe that $\Phi_V(V_r)$ is given by $\tilde h^2 + |\tilde\lambda|^2 \le r^2$ and
$0\le \tilde \tau \le \partial_{\tilde H}(S\circ\phi^{-1})(\tilde h,\tilde\lambda)$, 
and each point 
$$
\Phi_U (\tilde \lambda, \psi, \sigma \eta t_0^{n+1}, Q) \in \Phi_U(\partial_\sigma U_r) \subset
D^{n-1}\times T^{n-1}\times D^2, \quad 
H_{can}^{\tilde\lambda}(\sigma \eta t_0^{n+1}, Q)^2 + |\tilde\lambda|^2 < r^2,\ Q<0,
$$
$\sigma\in\{\pm1\}$, is glued with the corresponding point 
$(\tilde h,\tilde\tau,\tilde \lambda,\tilde\psi) \in 
\Phi_V(\partial_\sigma U_r)$, where 
\begin{equation} \label {eq:glue}
\tilde h = H_{can}^{\tilde\lambda}(\sigma \eta t_0^{n+1},Q), \quad
\begin{array}{lll}
\tilde\tau=0, & \tilde\psi_j=\psi_j & \mbox{if } \sigma=-1, \\
\tilde\tau = \partial_{\tilde H}(S\circ\phi^{-1})(\tilde h,\tilde \lambda), & 
\tilde\psi_j=\psi_j + \partial_{\tilde J_j}(S\circ\phi^{-1})(\tilde h,\tilde \lambda) & \mbox {if } \sigma=+1.
\end{array}
\end{equation}

{\em Case 2:} $\mathcal L\setminus\mathcal O$ is disconnected, thus $n$ is odd 
(and $\eta=-1$), $\ell=2$.
Similarly to Case 1, 
the closed tubular neighborhood $\{\tilde H^2 + \sum_{j=1}^{n-1}\tilde J_j^2 \le r^2\}$ of $\mathcal L$ can be obtained by gluing 
$\Phi_U(U_r)$ with 
$\Phi_V(V_r^1)$ and $\Phi_V(V_r^2)$ along the positive and negative boundaries $\Phi_U(\partial_\pm U_r)$, $\Phi_U(\partial_\pm' U_r)$ 
via the gluing map $\Phi_V\circ \Phi_U^{-1}|_{\Phi_U(\partial_\pm U_r) \cup \Phi_U(\partial_\pm' U_r)}$. One shows similarly to Case 1 that the piece $\Phi_U(U_r)$ and the symplectic structure on it are uniquely determined by the mapping germ $J\circ\phi^{-1}$, while 
the pieces $\Phi_V(V_r^i)$ and the gluing mapping are uniquely determined by the permutation $\beta\in\Sym_2$ (distinguishing Subcases 1 and 2) and by the generating functions $S_i\circ\phi^{-1}$, $i=1,2$.

Step 5.
Define the phase space of the desired integrable system as a topological manifold, which is glued from two symplectic manifolds with boundaries, $U_r$ and $V_r$, along their boundaries via the gluing map \eqref {eq:glue}. Since the Hamiltonian vectorfield $X_{\tilde H}$ is transversal to the glued boundary, moreover the gluing preserves the momentum map and the restrictions of the symplectic 2-from $\omega$ to the boundaries, it follows from \cite{Lazutkin1993} that the result of such a gluing is smooth (moreover, it is real-analytic in the real-analytic case). 
Moreover, the result of such a gluing is unique in the following sense: if pieces $\widetilde U_r, \widetilde V_r$ are copies of $U_r, V_r$ (respectively), then the identity maps $U_r\to\widetilde U_r$, $V_r\to\widetilde V_r$ induce a symplectomorphism between the results of gluing \cite{Lazutkin1993}. Therefore, our model system is well-defined, and every real-analytic unfolded $A_n$ semi-global singularity is 
symplectically equivalent to such a model.
\end{proof}

\appendix
\section {Proof of Theorem \ref {thm:An:loc}} \label {app:loc}

In this section, we give a proof of Theorem \ref {thm:An:loc} about a (left-right) local normal form. More specifically, there is the following Eliasson--Vey-type result.

\begin{theorem}[Symplectic normal form for unfolded $A_n$ singular points] \label{thm:An:loc:} Let $P$ be an unfolded $A_n$ singular point of a real-analytic integrable system $\mathcal F  = (F_1,\dots, F_n)\colon\, M \to \mathbb R^n$, $n\ge1$. Then there exist real-analytic coordinates $\Phi = (\tilde \lambda, \tilde \mu,\tilde x,\tilde y)$ centered at $P$ such that
$$
\tilde H = \eta \tilde{p}^2 + \tilde{q}^{n+1} + \sum_{j=1}^{n-1} \tilde{\lambda}_j \tilde{q}^j \quad \mbox{ and } \quad \tilde{J}_j = \tilde \lambda_j, \qquad 1\le j\le {n-1}, \qquad \eta=\pm1,
$$
are real-analytic functions of $(F_1,\dots,F_n)$ and the symplectic structure has the canonical form
$$
(\Phi^{-1})^*\omega = \sum_{j=1}^{n-1} \textup{d} \tilde{\lambda}_j \wedge \textup{d} \tilde \mu_j + \textup{d} \tilde{p} \wedge \textup{d} \tilde{q} .
$$
The functions $\tilde H$ and $\tilde{J}_1,\dots,\tilde{J}_{n-1}$ are rigid in the following sense: they are uniquely defined up to simultaneously changing the sign of the functions $\tilde J_{2i-1}$, $i=1,\dots,m$, when $n=2m+1\ge3$ is odd, or up to changing the sign of $\tilde H$ if $n=1$ and $\eta=-1$.
\end{theorem}
 
\begin{proof}
The first assertion of the theorem follows from \cite[Theorem 3]{VarchenkoGivental1982}; see also \cite[Theorem 6]{ColindeVerdiere2003} and \cite{Garay2004}. We give a more detailed proof for any $n\ge1$ below.
Note that an unfolded $A_n$ singularity is infinitesimally non-degenerate (see \cite[Theorem 5.25]{Zoladek2006} for a proof).

Step 1. 
By \eqref {eq:Phi1}--\eqref {eq:Morse:type:}, we can assume that 
$$
\mathcal F=(H,J_1,\dots,J_{n-1})=(\eta p^2+q^{n+1}+\sum_{j=1}^{n-1} \lambda_j q^j,\lambda_1,\dots,\lambda_{n-1}), \qquad 
\omega= \sum_{j=1}^{n-1} \textup{d}\lambda_j\wedge\textup{d}\mu_j + \textup{d}\alpha,
$$
where $\alpha$ is the germ at $0$ of a real-analytic 1-form in the $(n+1)$-space $(p,q,\lambda)$ such that $\textup{d}\alpha|_{\lambda=\const} = \tilde g(p,q,\lambda) \textup{d}p\wedge\textup{d}q$, $\tilde g(0,0,0)>0$.
Let $I_k(h,\lambda)$, $1\le k\le n$, be the (local, complex-valued) action variables corresponding to some basic cycles $\gamma_{h,\lambda,k}$, $1\le k\le n$, of the first homology group of the complex leaf 
$$
L_{h,\lambda}=\{(p,q)\in\mathbb C^2 \mid H(p,q,\lambda)=h\}, 
\quad (h,\lambda)\in\mathbb C^n\setminus\Sigma^{\mathbb C}
$$
(which is a genus-$[\frac{n}{2}]$ surface with $\frac{3-(-1)^n}2$ punctures), where
$$
\Sigma^{\mathbb C} := \{(h,\lambda)\in\mathbb C^n \mid  
\mbox{the polynomial } q^{n+1} + \sum_{j=1}^{n-1} \lambda_j q^j - h \mbox{ has a multiple root}\}.
$$
Clearly, we can assume that the cycles $\gamma_{h,\lambda,k}$, $1\le k\le n$, lie on $\{\mu=0\}$, thus 
$$
(h,\lambda)\mapsto (I_1(h,\lambda),\dots,I_n(h,\lambda)) \in \CC^n, \quad (h,\lambda)\in\mathbb C^n\setminus\Sigma^{\mathbb C},
$$
is the {\em period mapping} of the 1-form $\alpha$ (see \cite[Sec.~1]{VarchenkoGivental1982}), where we regard $(I_1(h,\lambda),\dots,I_n(h,\lambda))$ as an element of $H^1(L_{h,\lambda}, \CC) \cong \CC^n$.

Step 2.
Let us show that the period mapping is {\em infinitesimally non-degenerate} in the sense of \cite[Sec.~1]{VarchenkoGivental1982}, i.e., 
\begin{equation} \label {eq:lim}
\lim_{h\to0} \det\frac{\partial (I_1,\dots,I_n)}{\partial(h,\lambda)}|_{\lambda=0} \ne0
\end{equation}
(where the Jacobi matrix is written in a locally constant basis). For this, we observe that 
$$
I_k(h,\lambda) 
= \frac1{\pi} \int_{q_{k}(h,\lambda)}^{q_{k+1}(h,\lambda)} 
\tilde g(p(q,h,\lambda),q,\lambda) p(q,h,\lambda) \textup{d}q, \quad 1\le k\le n, \quad 
p(q,h,\lambda) := \sqrt{\eta(h-q^{n+1}-\sum_{j=1}^{n-1}\lambda_jq^j)},
$$
where $q_k(h,\lambda)$, $1\le k\le n+1$, are roots of the polynomial $q^{n+1}+\sum_{j=1}^{n-1}\lambda_jq^j - h$, $(h,\lambda)\in\mathbb C^n\setminus\Sigma^{\mathbb C}$.
Therefore, the Jacobi matrix in \eqref {eq:lim} has the following entries:
\begin{align*}
\frac{\partial I_k}{\partial h}(h,0) = \frac {\eta } {2\pi} \int_{q_{k}(h,0)}^{q_{k+1}(h,0)} 
\frac{\tilde g(\sqrt{\eta(h-q^{n+1})},q,0) \textup{d}q} {\sqrt{\eta(h-q^{n+1})}}  + \dots
= \frac {\eta \tilde g(0,0,0)} {2\pi} \int_{q_{k}(h,0)}^{q_{k+1}(h,0)} 
\frac{\textup{d}q} {\sqrt{\eta(h-q^{n+1})}}  + \dots, 
\\
\frac{\partial I_k}{\partial \lambda_j}(h,0) = - \frac {\eta} {2\pi} \int_{q_{k}(h,0)}^{q_{k+1}(h,0)} 
\frac{\tilde g(\sqrt{\eta(h-q^{n+1})},q,0) q^j \textup{d}q} {\sqrt{\eta(h-q^{n+1})}}  + \dots
= - \frac {\eta\tilde g(0,0,0)} {2\pi} \int_{q_{k}(h,0)}^{q_{k+1}(h,0)} 
\frac{q^j \textup{d}q} {\sqrt{\eta(h-q^{n+1})}}  + \dots,
\end{align*}
$1\le k\le n$, $1\le j\le n-1$, where 
the unwritten terms are of order $o(h^{\frac{1}{n+1}-\frac12})$ and $o(h^{\frac{j+1}{n+1}-\frac12})$, respectively. Hence, the determinant of the Jacobi matrix has the form 
$-(-\eta \frac {\tilde g(0,0,0)} {2\pi})^n D + o(1)$ as $h\to0$, where
$D = \det (\int_{q_{k+1}(1,0)}^{q_{k+2}(1,0)} 
\frac{q^j \textup{d}q} {\sqrt{\eta(1-q^{n+1})}})_{0\le k,j<n}$.
It is known that $D\ne0$ (see \cite[Theorem 5.25]{Zoladek2006} for a proof, see also \cite[below Theorem 8]{Varchenko1987}, \cite[Examples 37.2.A]{Arnold2012} for universal unfoldings of quasi-homogeneous singularities). 
Therefore, the limit in \eqref {eq:lim} equals $-(-\eta \frac {\tilde g(0,0,0)} {4\pi})^n D \ne 0$, since $\tilde g(0,0,0)>0$ (the latter inequality holds in our case, since 
$\omega=\sum_{j=1}^{n-1}\textup{d}\lambda_j\wedge\textup{d}\mu_j + \textup{d}\alpha$ is non-degenerate).

Step 3. 
Due to \cite[Theorem 3]{VarchenkoGivental1982} (see also \cite[Theorem 6]{ColindeVerdiere2003} and \cite{Garay2004}), any infinitesimally non-degenerate period mapping is stable, i.e., 
the period mappings of any two 1-forms close to $\alpha$ are equivalent in the following sense: there exists a complex-analytic diffeomorphism germ $\phi^\CC \colon (\mathbb C^n,0) \to (\mathbb C^n,0)$ close to the identity that preserves $\Sigma^{\mathbb C}$ and respects the period mappings\footnote{see~\cite[Sec.~2]{VarchenkoGivental1982} for more precise definitions of equivalence and stability for period mappings.}. 
This implies that the period mappings of the 1-forms $\alpha$ and 
$$
\tilde\alpha=p \textup{d}q
$$
are equivalent in the above sense.
Indeed: consider the 1-parameter family of 1-forms $\alpha_t=(1-t)\alpha+t\tilde\alpha$, $0\le t\le1$. Since their period mappings $I_t=(I_{1,t},\dots,I_{n,t})$ are stable (see above), there exists a 1-parameter family of complex-analytic diffeomorphism germs $\phi_t^\CC:(\CC^n,0)\to(\CC^n,0)$ such that 
\begin{equation*} 
\phi_0^\CC = \mathrm{id}, \qquad 
(I_{1,t},\dots,I_{n,t}) \circ\phi_t^\CC= (I_1,\dots,I_n), \quad 0\le t\le 1.
\end{equation*}
In particular, $\phi^\CC=\phi_1^\CC$ transforms the period mapping of $\tilde\alpha$ to that of $\alpha$, i.e.,
\begin{equation} \label {eq:phi}
(\tilde I_1,\dots,\tilde I_n) \circ \phi^\CC = (I_1,\dots,I_n).
\end{equation}
Here $\tilde I = (\tilde I_1,\dots,\tilde I_n) = (I_{1,t},\dots,I_{n,t})|_{t=1}$ denotes the period mapping of $\tilde\alpha$.

Step 4.
But infinitesimally non-degenerate period mappings are {\em non-dege\-ne\-rate} \cite[Sec.~1]{VarchenkoGivental1982}, i.e.,
\begin{equation} \label {eq:det}
\det\frac{\partial(I_1,\dots,I_n)}{\partial(h,\lambda)}\ne0
\end{equation} 
at each point $(h,\lambda)\in\mathbb C^n\setminus\Sigma^{\mathbb C}$ sufficiently close to the origin (see \cite[Theorem~1]{VarchenkoGivental1982}).
Recall that, on the swallowtail domain 
$$
D_{r} := 
\{(h,\lambda)\in\mathbb R^n \mid \lambda_{n-1} > - r, \
\exists q_1 < q_2 < \dots < q_{n+1} \ \forall q \in\R \ \prod_{i=1}^{n+1}(q-q_i) = q^{n+1} + \sum_{j=1}^{n-1} \lambda_j q^j - h\}
$$
(see \eqref {eq:swallow}), the period mapping of any real-analytic 1-form $\alpha$ with $\textup{d}\alpha|_{\{\lambda=0\}}/(\textup{d}p \wedge\textup{d}q)>0$ has the form
\begin{equation} \label {eq:P}
P := (I_1,\dots,I_n)|_{D_{r}}\colon D_{r}\to \Delta \approx \R^n_{>0},
\end{equation}
see \eqref {eq:I:}.
It follows from \eqref {eq:swallow} and \eqref {eq:det} that the mapping \eqref {eq:P} is an immersion if $r>0$ is small. 
This and~\eqref{eq:phi} imply that $\phi^\CC$ is real-analytic.

Thus, we have a real-analytic diffeomorphism germ $\phi:(\R^n,0)\to(\R^n,0)$ preserving the swallowtail domain $D_r$, 
such that $\phi|_{D_{r}}$ respects the action variables $I_k|_{D_{r}}$ and $\tilde I_k|_{D_{r}}$ 
corresponding to the vanishing cycles $\gamma_{h,\lambda,k}$, $1\le k\le n$, of the integrable system germs (at the origin) $(\mathbb R^{2n},\omega,\mathcal F)$ and $(\mathbb R^{2n},\tilde\omega,\mathcal F)$, where 
$$
\tilde\omega
= \sum_{j=1}^{n-1} \textup{d}\lambda_j\wedge\textup{d}\mu_j + \textup{d}\tilde\alpha
= \sum_{j=1}^{n-1} \textup{d}\lambda_j\wedge\textup{d}\mu_j + \textup{d} p \wedge\textup{d}q.
$$
This implies that the diffeomorphism germ $\phi:(\R^n,0)\to(\R^n,0)$ can be lifted to a real-analytic foliation-preserving symplectomorphism germ $\Phi \colon (\mathbb R^{2n},0) \to (\mathbb R^{2n},0)$ for these two systems (by the same arguments as in \cite[Proposition~5.3]{Bolsinov2018} where $2\pi$-periodicity of the Hamiltonian flow generated by $\lambda_j$ is in fact inessential, $1\le j\le n-1$). 
This means that $\mathcal F\circ\Phi = \phi\circ\mathcal F$ and 
$(\Phi^{-1})^* \omega = \tilde\omega$.
It is left to put 
$(\tilde\lambda,\tilde\mu,\tilde p,\tilde q):=(\lambda,\mu,p,q)\circ\Phi$ and
$(\tilde H,\tilde J_1,\dots,\tilde J_{n-1}) := \phi\circ\mathcal F$.

Step 5.
For proving that the functions $\tilde H$ and $\tilde J_1,\dots,\tilde J_{n-1}$ are uniquely defined up to some involution, let us suppose that $\alpha=\tilde\alpha$ and show that $\phi$ is either the identity or the corresponding involution. 
We will prove this under the assumption $n+\eta>0$, which means that there exists at least one ``real'' vanishing cycle $\gamma_{h,\lambda,k}$ for $(h,\lambda) \in D_r$ (the remaining case 
$n=1$ and $\eta=-1$ is well-known and can be treated separately).

Since $\Phi$ is real, it sends every ``real'' vanishing cycle $\gamma_{h,\lambda,k}$ of $L_{h,\lambda}$ (see \eqref {eq:real}), $(h,\lambda) \in D_r$, to a ``real'' vanishing cycle of $L_{\phi(h,\lambda)}$, for any sufficiently small $r>0$. Let us show that the complexified symplectomorphism germ $\Phi^{\mathbb C}$ sends ``imaginary'' vanishing cycles (see \eqref {eq:real}) to ``imaginary'' ones. The swallowtail domain 
$$
D := 
\{(h,\lambda)\in\mathbb R^n \mid \exists q_1 < q_2 < \dots < q_{n+1} \ \forall q\in\R \ 
\prod_{i=1}^{n+1}(q-q_i) = q^{n+1} + \sum_{j=1}^{n-1} \lambda_j q^j - h \}
$$
is characterized by the following property: the number of ``real'' vanishing cycles is maximal possible (and positive, since $n+\eta>0$) if and only if $(h,\lambda) \in D$.
Hence, the diffeomorphism germ $\phi$ preserves the swallowtail domain $D$. 
This implies that $\Phi^{\mathbb C}$ sends ``imaginary'' vanishing cycles to ``imaginary'' ones, for $(h,\lambda) \in D_r$.

Step 6.
It follows that the mapping $\Phi^{\mathbb C}$ either preserves or reverses numbering of the vanishing cycles $\gamma_{\phi(h,\lambda),k}$ of $L_{\phi(h,\lambda)}$, for $(h,\lambda) \in D_r$ (since pairwise intersection numbers $\mod2$ of the vanishing cycles form the identity $n\times n$-matrix). We have two cases.

{
{\em Case 1:} the symplectomorphism $\Phi^{\mathbb C}$ preserves numbering of the vanishing cycles for $(h,\lambda) \in D_r$, so it sends every vanishing cycle $\gamma_{h,\lambda,k}$ of $L_{h,\lambda}$ to the corresponding vanishing cycle $\gamma_{\phi(h,\lambda),k}$ of $L_{\phi(h,\lambda)}$, $(h,\lambda) \in D_r$. Then it preserves orientation of each vanishing cycle (since $\Phi^{\mathbb C}$ is a symplectomorphism and, hence, preserves the action variable $I_k$ corresponding to $\gamma_{h,\lambda,k}$). 
As we will show (see below), $\phi$ is the identity in this case.

{\em Case 2:} the symplectomorphism
$\Phi^{\mathbb C}$ reverses numbering of the vanishing cycles for $(h,\lambda) \in D_r$, so it sends every vanishing cycle $\gamma_{h,\lambda,k}$ of $L_{h,\lambda}$ to the vanishing cycle $\gamma_{\phi(h,\lambda),n+1-k}$ of $L_{\phi(h,\lambda)}$, $(h,\lambda) \in D_r$. Then $n$ is odd (since otherwise ``real'' cycles are sent to ``imaginary'' ones, see \eqref {eq:real}). Since, for odd $n$, the canonical involution $\Psi$ (see \eqref {eq:Psi}) also reverses numbering of the vanishing cycles and satisfies the relation \eqref {eq:Psi:psi} for the involution $\psi$ (see \eqref {eq:psi}), it follows that the symplectomorphism $\Phi\circ\Psi$ 
preserves numbering of the cycles and satisfies the relation $\mathcal F_{can}\circ(\Phi\circ\Psi) = (\phi\circ\psi)\circ\mathcal F_{can}$. It follows from Case 1 that $\phi\circ\psi$ is the identity, hence $\phi=\psi$, as required.

It remains to prove that $\phi$ is the identity in Case 1. 
Since $\Phi^{\mathbb C}$ is a symplectomorphism preserving numbering of vanishing cycles, it preserves action variables, so it preserves the immersion $P$ in \eqref {eq:P}, i.e.,
\begin{equation} \label {eq:PP}
P\circ\phi^{\mathbb C}=P
\end{equation}
on some $D_{r_1}$.
}
But this immersion admits a continuous extension 
$$
P_1 \colon \overline{D_{r}}\to \overline{\Delta} \approx \R^n_{\ge0}
$$
such that $P_1(0,0)=(0,0)$ and $P_1$ sends the branches
$$
\partial_k D_{r} := 
\{ (h,\lambda)\in\R^n \mid 
\lambda_{n-1}>-r,\ \exists q_1 < \dots < q_k = q_{k+1}< \ldots < q_{n+1} \ \forall q\in\R\ 
\prod_{i=1}^{n+1}(q-q_i) = q^{n+1} + \sum_{j=1}^{n-1} \lambda_j q^j - h  \}, 
$$
$1\le k\le n$, $\partial_k D_{r}\subseteq\partial D_{r}\subseteq\Sigma_r$, of the bifurcation diagram 
$$
\Sigma_r := \{(h,\lambda)\in\mathbb R^n \mid  
\lambda_{n-1}>-r,\ 
\mbox{the polynomial } q^{n+1} + \sum_{j=1}^{n-1} \lambda_j q^j - h \mbox{ has a multiple real root} \}
$$ 
(see \eqref {eq:Sigma} and \eqref {eq:swallow}) to the corresponding open facets 
$$
\partial_k\Delta \approx \R_{>0}^{k-1} \times \{0\} \times \R_{>0}^{n-k}, \qquad 1\le k\le n,
$$
of the positive coordinate octant $\Delta\approx\R^n_{>0}$ (see \eqref {eq:I:}).
This implies\footnote{We will give a topological proof of injectivity of the immersion
$P=(I_1,\dots,I_n)|_{D_r} : D_r\to\Delta\approx\R^n_{>0}$ 
(see \eqref {eq:P}) on some $D_{r_1}$.
We have $\delta := \inf(|I_1|+\dots+|I_n|)|_{\partial D_{r}\cap\{\lambda_{n-1}=-r\}}>0$. Consider the ``tetrahedral'' domain $\Delta_\delta := \{ (t_1,\dots,t_n)\in\Delta \mid |t_1|+\dots+|t_n| <\delta\}$ (see \eqref {eq:I:}). 
It is known that the number of preimages (counted with multiplicities)
of any point $t_0 \in\mathbb R^n\setminus P_1(\partial D_{r})$ under a continuous mapping $P_1:\overline{D_{r}}\to\mathbb R^n$ equals index of this point relatively the $(n-1)$-dimensional spheroid $P_1|_{\partial D_{r}}$ (this follows from the theorem on the sum of indices of zeros of a vector field on a bounded domain in $\R^n$, by applying it to the vector field $(h,\lambda)\mapsto P(h,\lambda) - t_0$ on $D_{r}$). Note that all multiplicities of points in $P^{-1}(t_0)$ are equal to $\textup{sgn} \, \det \frac{\partial (I_1,\dots,I_n)}{\partial (h,\lambda)} = 1$ ($P$ being an immersion). But for any point $t_0\in\Delta_\delta$ the index equals 1. We conclude that this point has a unique preimage under $P$. Thus $P$ 
induces a homeomorphism $P^{-1}(\Delta_\delta) \to \Delta_\delta$; in particular, $P$
is injective on $P^{-1}(\Delta_\delta)\subseteq D_{r}$ which obviously contains some $D_{r_1}$.}
that $P$ is injective on $D_{r_1}$ for some $r_1>0$. Since \eqref {eq:PP} holds on $D_{r_1}$, we conclude that $\phi$ is the identity on some $D_{r_2}$, and hence everywhere by uniqueness of analytic continuation.

Theorem \ref {thm:An:loc:} (which is the same as Theorem \ref {thm:An:loc}) is proved.

We remark that the equivariance property of the local action mapping $I=(I_1,\dots,I_n)$ (see Proposition \ref {pro:equivar}) under the (real) involutions $\psi$ and $\chi$ extends to equivariance under the (complex, non-real) diffeomorphisms 
$$
\psi_c : (h,\lambda) \mapsto (c^{2(n+1)} h,\dots,c^{2(n+1-j)}\lambda_j,\dots),
\qquad
\chi_c : (I_1,\dots,I_n) \mapsto (c^{n+3}I_1,\dots,c^{n+3}I_n),
$$
$c\in\CC^*=\CC\setminus\{0\}$, where $\psi_c$ lifts to the diffeomorphism
$$
\Psi_c : (\lambda,\mu,p,q) \mapsto (\dots,c^{2(n+1-j)}\lambda_j,\dots, c^{2j+1-n}\mu_j, \dots, c^{n+1}p,c^2q).
$$
The diffeomorphism $\Psi_c$ is a symplectomorphism if and only if $c\in\{a^k\mid k\in\mathbb Z\}$ where
$a:=e^{\pi i/(m+2)}$ if $n=2m+1$,
$a:=e^{2\pi i/(2m+3)}$ if $n=2m$. 
Thus, if $n=2m$ is even then, in the complex setting, the functions $\tilde H$ and $\tilde J_1,\dots,\tilde J_{n-1}$ are not uniquely defined, since the canonical coordinate change $\Psi_{a}$
transforms 
$(\tilde H,\tilde J_1,\dots,\tilde J_{n-1})$ to
$\psi_{a}(\tilde H,\tilde J_1,\dots,\tilde J_{n-1}) 
= (e^{-8\pi	i/(2m+3)}\widetilde H,\dots,e^{-4\pi i(j+2)/(2m+3)}\tilde J_j,\dots)$.  
\end{proof}

\section {Proof of Theorem \ref {thm:An:semiloc} and Corollary \ref {cor:An:semiloc:class}} \label {app:semiloc}

\begin{theorem} [Symplectic normal form for unfolded $A_n$ singular orbits]
\label{thm:An:semiloc:} 
Let $\mathcal O$ be a compact unfolded $A_n$ singular orbit of a real-analytic integrable system $\mathcal F = (F_1, \dots, F_n) \colon M^{2n} \to \mathbb R^n$, $n\ge2$. 
Consider real-analytic canonical coordinates $\Phi = (\tilde J = \tilde\lambda, \tilde \mu, \tilde p,\tilde q)$ and the functions $(\tilde H,\tilde J_1,\dots,\tilde J_{n-1}) = \phi\circ \mathcal F$ in a neighborhood of some point $P \in \mathcal O$ as in Theorem \ref {thm:An:loc}.
Let $J_j(F_1,\dots,F_n) = J_j\circ\mathcal F$, 
$1\le j\le n-1$, be the first integrals of the system on a neighbourhood of $\mathcal O$ 
generating $2 \pi$-periodic Hamiltonian flows in a neighbourhood of $\mathcal O$ such that 
$J_j(\mathcal F(\mathcal O)) = 0$ and 
$\det(\partial_{\tilde J_i}(J_j\circ\phi^{-1}))_{1\le i,j\le n-1} > 0$ at the origin (they exist, due to \cite{Zung2000}).
Then the functions $\tilde H,\tilde J_1,\dots,\tilde J_{n-1}$ and the symplectic structure $\omega$ have the following (semi-local) normal form near $\mathcal O$:
\begin{equation*} 
\tilde H = \eta P^2 + Q^{n+1} + \sum_{j=1}^{n-1} \lambda_j Q^j, \quad \tilde J_j = \tilde\lambda_j, \qquad
\omega = \sum_{j=1}^{n-1} \textup{d} (J_j\circ\phi^{-1}(\tilde H,\tilde J_1,\dots,\tilde J_{n-1})) \wedge \textup{d} \psi_j + \textup{d}P \wedge \textup{d}Q
\end{equation*}
in some real-analytic coordinates 
$(\tilde J=\lambda, \psi,P,Q) : U \to D^{n-1} \times T^{n-1}\times D^2$ 
on a neighbourhood $U$ of $\mathcal O$, in which $\mathcal O=(0,\dots,0)\times T^{n-1}\times(0,0)$.
The tuple of functions $J\circ\phi^{-1} = (J_1\circ\phi^{-1},\dots,J_{n-1}\circ\phi^{-1})^T$ is rigid in the following sense: 
it is uniquely defined if $n=2$;
it is uniquely defined up to replacing it with $AJ\circ\phi^{-1}$, $A\in SL(n-1,\mathbb Z)$, if $n\ge4$ is even;
it is uniquely defined up to replacing it with $AJ\circ\phi^{-1}$ or $AJ\circ\phi^{-1}\circ\psi$, $A\in SL(n-1,\mathbb Z)$, if $n\ge3$ is odd (for the involution $\psi$ in \eqref {eq:psi}).
\end{theorem} 

\begin{proof}
To prove the semi-local normal form, it suffices to define new coordinates $\psi,P,Q$ as follows. Consider local coordinates
$\tilde \lambda, \tilde \mu,\tilde p, \tilde q$ as in Theorem~\ref{thm:An:loc}.
On the embedded disk $\tilde D^{n+1}$ given by $\tilde\mu = 0$, the new coordinates $P,Q$ simply coincide with $\tilde p, \tilde q$. These coordinates $P,Q$ on $\tilde D^{n+1}$ are then extended to a neighbourhood of the unfolded $A_n$ singular orbit $\mathcal O$ using the $2\pi$-periodic flows 
\begin{equation} \label {eq:flow}
\Phi_t=\phi_{J_1\circ\mathcal F}^{t_1}\circ\dots\circ\phi_{J_{n-1}\circ\mathcal F}^{t_{n-1}}, \qquad t=(t_1,\dots,t_{n-1})\in\R^{n-1},
\end{equation}
of the first integrals $J_1\circ\mathcal F,\dots,J_{n-1}\circ\mathcal F$, by making them invariant under these flows.
The remaining angle coordinates $\psi=(\psi_1,\dots,\psi_{n-1})$ are
simply the ``time'' $t=(t_1,\dots,t_{n-1})$ needed to reach a given point from $\tilde D^{n+1}$ using the flows \eqref {eq:flow} of $J_1\circ\mathcal F,\dots,J_{n-1}\circ\mathcal F$.

The rigidity assertion follows directly from Theorem \ref {thm:An:loc} and the existence of $2\pi$-periodic first integrals $J\circ\mathcal F=(J_1,\dots,J_{n-1})^T\circ\mathcal F$ in a neighbourhood of a compact unfolded $A_n$ singular orbit $\mathcal O$ such that $J\circ\mathcal F=\ldbrack J\circ\mathcal F\rdbrack$, as we now show.
First we will give a scketch of the proof.
Suppose that $n$ is even, then the functions $\tilde H$ and $\tilde J_1,\dots,\tilde J_{n-1}$ are uniquely defined, due to the rigidity property in Theorem \ref {thm:An:loc}.
Observe that $2\pi$-periodic first integrals $J\circ\mathcal F=(J_1,\dots,J_{n-1})^T\circ\mathcal F$ are action variables of the system. Further, by Remark \ref {rem:semiloc:}, if the $SL(n-1,\Z)$-orbits of the corresponding mapping germs coincide then the systems are symplectically equivalent near unfolded $A_n^\eta$ singular orbits. 
Since the conditions $J(0,\dots,0) = 0$,
$\det(\partial_{\tilde J_i}(J_j\circ\phi^{-1}))_{1\le i,j\le n-1} > 0$ at the origin determine the action variables $[J\circ\mathcal F] = [J\circ\phi^{-1} (\tilde H, \tilde J_1,\dots,\tilde J_{n-1})]$ unambiguously, the germ of $[J\circ\phi^{-1}]$ at $(0,\dots,0)$ is indeed a symplectic invariant.
If $n$ is odd, then the diffeomorphism germ $\phi$ is defined up to replacing it with $\psi\circ\phi$. But still we can determine the tuple of functions $\ldbrack J\circ\mathcal F\rdbrack$ unambiguously as one of the mapping germs
$[J\circ\phi^{-1}]$ and $[J\circ\phi^{-1}\circ\psi]$ that is ``not smaller'' with respect to lexicographic order. Then such a germ $\ldbrack J\circ\phi^{-1}\rdbrack$ is indeed a symplectic invariant.

Now let us give a more detailed proof of the rigidity assertion.
Suppose that two integrable systems $(M^i,\omega^i,\mathcal F^i)$ are symplectically equivalent on some neighborhoods $U^i$ of unfolded $A_n^\eta$ singular orbits $\mathcal O^i$ ($i=1,2$).
This means that there exist a symplectomorphism $\widetilde\Phi:U^1\to U^2$ and a diffeomorphism germ $\widetilde\phi:(\R^n,\mathcal F^1(\mathcal O^1))\to (\R^n,\mathcal F^2(\mathcal O^2))$ such that 
$\widetilde\phi\circ\mathcal F^1 = \mathcal F^2 \circ \widetilde\Phi$.
On the other hand, we have a local (in a neighborhood $U(P^i)$ of $P^i$) symplectomorphism $(\Phi^2)^{-1}\circ\Phi^1:U(P^1)\to U(P^2)$ and a diffeomorphism germ $(\phi^2)^{-1}\circ\phi^1:(\R^n,\mathcal F^1(\mathcal O^1))\to (\R^n,\mathcal F^2(\mathcal O^2))$ such that 
$(\phi^2)^{-1}\circ\phi^1\circ\mathcal F^1 = \mathcal F^2 \circ (\Phi^2)^{-1}\circ\Phi^1$.
The rigidity property in Theorem \ref{thm:An:loc} implies that 
\begin{equation} \label {eq:phi:i}
\begin{gathered}
\widetilde\phi = (\phi^2)^{-1}\circ \phi^1 \qquad \mbox{if $n$ is even}; \\
\widetilde\phi = (\phi^2)^{-1}\circ \phi^1 \ \mbox{ or } \
\widetilde\phi = (\phi^2)^{-1}\circ \psi \circ \phi^1 
\quad \mbox{if $n$ is odd}.
\end{gathered}
\end{equation}
By Theorem \ref{thm:An:semiloc}, we can write every system (possibly, in smaller neighborhoods $\widetilde U^i\subset U^i$ of $\mathcal O^i$, $\widetilde\Phi(\widetilde U^1)=\widetilde U^2$) in the standard form \eqref {eq:semiloc}, with some mapping germs $J^i\circ(\phi^i)^{-1}=(J^i_1,\dots,J^i_{n-1})^T\circ(\phi^i)^{-1}: (\R^n,0)\to(\R^{n-1},0)$, $i=1,2$.
Since the functions 
$J^2_1\circ\mathcal F^2,\dots,J^2_{n-1}\circ\mathcal F^2$ generate a (free) Hamiltonian $T^{n-1}$-action on $\widetilde U^2$, it follows that the functions 
$J^2_1\circ\mathcal F^2\circ\widetilde\Phi,\dots,J^2_{n-1}\circ\mathcal F^2\circ\widetilde\Phi$
generate a (free) Hamiltonian $T^{n-1}$-action on $\widetilde U^1$.
But the functions 
$J^1_1\circ\mathcal F^1,\dots,J^1_{n-1}\circ\mathcal F^1$ also generate a (free) Hamiltonian $T^{n-1}$-action on $\widetilde U^1$. 
It follows that there exists an integer matrix $A=(a^i_j)\in SL(n-1,\Z)$ such that 
$J^1_i = \sum_{j=1}^{n-1} a^i_j J^2_j \circ\widetilde\phi$, $1\le i\le n-1$ (otherwise we would have an effective $\mathcal F^1$-preserving $T^n$-action near the orbit $\mathcal O^1$, which would imply that this orbit is non-degenerate of elliptic type \cite[Theorem 3.4]{Kudryavtseva2020}, which is possible only in the case of $n=1$ and $\eta=+1$, but this case contradicts the assumption of Corollary).
Thus, the $SL(n-1,\Z)$-orbits of the mapping germs $J^1$ and $J^2\circ\widetilde\phi$ coincide.
Taking into account \eqref {eq:phi:i}, we conclude that either 
the $SL(n-1,\Z)$-orbits of the mapping germs 
$J^1\circ(\phi^1)^{-1}$ and $J^2\circ(\phi^2)^{-1}$ coincide,
or $n$ is odd and the $SL(n-1,\Z)$-orbits of the mapping germs 
$J^1\circ(\phi^1)^{-1}$ and $J^2\circ(\phi^2)^{-1}\circ\psi$ coincide.
\end{proof}

\begin{corollary} [Symplectic classification of unfolded $A_n$ singular orbits] \label {cor:An:semiloc:class:}
Under the hypotheses of Theorem \ref {thm:An:semiloc}, let $[J\circ\phi^{-1}]: (\R^n,0)\to(\R^{n-1},0)$ be the special mapping germ (see Remark \ref{rem:semiloc}) in the $SL(n-1,\Z)$-orbit of the mapping germ $J\circ\phi^{-1}: (\R^n,0)\to(\R^{n-1},0)$.
If $n$ is odd, take one of the mapping germs $[J\circ\phi^{-1}], [J\circ\phi^{-1}\circ\psi]: (\R^n,0)\to(\R^{n-1},0)$ which is ``not smaller'' with respect to lexicographic order, and denote it by $\ldbrack J\circ\phi^{-1}\rdbrack$; otherwise denote $[J\circ\phi^{-1}]$ by $\ldbrack J\circ\phi^{-1}\rdbrack$.
Then the real-analytic mapping germ $\ldbrack J\circ\phi^{-1}\rdbrack$ classifies the singularity at the orbit $\mathcal O$ up to real-analytic (semi-local, left-right) symplectic equivalence. 
In other words, two integrable systems $(M^i,\omega^i,\mathcal F^i)$ are symplectically equivalent near their unfolded $A_n^\eta$ singular orbits $\mathcal O^i$ ($i=1,2$, $n\ge2$) if and only if $\ldbrack J^1\circ(\phi^1)^{-1}\rdbrack = \ldbrack J^2\circ(\phi^2)^{-1}\rdbrack$.
Here $\Phi^i$ are symplectic coordinates centered at a point $P^i\in\mathcal O^i$ in which the $i$-th system has the local normal form 
$\phi^i\circ\mathcal F^i = \mathcal F^i_{can}\circ\Phi^i$ for some diffeomorphism germ $\phi^i:(\R^n,\mathcal F^i(\mathcal O^i))\to (\R^n,0)$.
\end {corollary}

\begin{proof}
It follows from Theorem \ref {thm:An:semiloc} (coinciding with Theorem \ref {thm:An:semiloc:}), Remark \ref {rem:semiloc} and \eqref {eq:Psi:psi} that the mapping germ $\ldbrack J\circ\phi^{-1}\rdbrack$ is a symplectic invariant of the singular orbit.

To show that a real-analytic germ 
$\ldbrack J\circ\phi^{-1}\rdbrack$ is the only invariant, and it admits arbitrary values, we now construct a `model' real-analytic integrable system with an unfolded $A_n$ singular orbit from such a germ.
Consider the phase space $C := \mathbb R^{n-1}\times T^{n-1}\times \R^2$ with coordinates
$(\tilde J = \tilde \lambda, \psi, P,Q)$ and
define the Hamilton function $\tilde H = H_{can}^{\tilde\lambda}(P,Q)$ on it.  
Suppose that we are given any mapping germ $f=(f_1,\dots,f_{n-1}):(\R^n,0)\to(\R^{n-1},0)$ such that 
$\det(\partial_{\tilde J_i}f_j)_{1\le i,j\le n-1} > 0$ at the origin. Then the desired integrable system is given by 
$$
(C,\ \omega ,\ \mathcal F), \qquad 
\omega = \sum_{j=1}^{n-1} \textup{d}f_j(\tilde H,\tilde J)
\wedge \textup{d}\psi_j + \textup{d}P\wedge \textup{d}Q, \quad 
\mathcal F = (\tilde H, \tilde{J});
$$
here the 2-form $\omega$ is non-degenerate, since the indicated determinant does not vanish.
By the construction, it is an integrable system with an unfolded $A_n$ singular 
orbit $P=Q= 0$, $\tilde\lambda = 0$. Moreover, due to Theorem \ref {thm:An:semiloc}, every real-analytic unfolded $A_n$ semi-local singularity is symplectically equivalent to such a model.
\end{proof}

\section {Acknowledgments} 

The author expresses her gratitude to A.V.\ Bolsinov and N.N.\ Martynchuk for insightful discussions that contributed to the preparation of this paper.

\section {Funding} 

The work was supported by the Russian Science Foundation under grant no.~24-71-10100
and performed at Lomonosov Moscow State University. The other part of the work was performed within the framework of the Program for the development of the Volga Region Scientific-Educational Centre of Mathematics (agreement no.~075-02-2024-1438).

\vspace{0.5cm}

\noindent
E.A. Kudryavtseva, Faculty of Mechanics and Mathematics, Moscow State University, Moscow 119991, Russia. \\ 
Moscow Center of Fundamental and Applied Mathematics, Moscow, Russia. \\ 
\textit{E-mail: eakudr@mech.math.msu.su} 

\end{document}